\documentclass{svmult}
\usepackage{fullpage}
\usepackage{url}

\usepackage{subfig}
\usepackage{mathptmx}       
\usepackage{helvet}         
\usepackage{courier}        
\usepackage{type1cm}        
\usepackage{makeidx}         
\usepackage{graphicx}        
\usepackage{multicol}        
\usepackage[bottom]{footmisc}

\usepackage{amsmath,amsfonts,amssymb}
\usepackage{cases}
\usepackage{color}
\usepackage{hyperref, enumerate}
\usepackage{cite}

\newcommand{\andf}{\quad\hbox{and}\quad}

\newcommand{\R}{\mathbb{R}}
\renewcommand{\r}{\mathbb{R}}

\newcommand{\T}{\mathbb{T}}

\renewcommand{\div}{\operatorname{div}}

\newtheorem{theo}{\bf Theorem}[section]

\newtheorem{coro}{\bf Corollary}[section]
\newtheorem{lem}{\bf Lemma}[section]
\newtheorem{rem}{\bf Remark}[section]
\newtheorem{defn}{\bf Definition}[section]

\newtheorem{prop}{\bf Proposition}[section]

\newcommand{\vc}[1]{{\bf #1}}
\newcommand{\vv}{\vc{v}}
\newcommand{\vx}{\vc{x}}
\newcommand{\eq}[1]{\begin{equation}
\begin{split}
#1
\end{split}
\end{equation}}
\newcommand{\eqh}[1]{\begin{equation*}
\begin{split}
#1
\end{split}
\end{equation*}}
\newcommand{\vr}{\varrho}
\newcommand{\lr}[1]{\left( #1 \right)}
\newcommand{\ep}{\varepsilon}
\newcommand{\dx}{{\rm d} {x}}
\newcommand{\dt}{{\rm d}\,t }
\newcommand{\dd}{{\rm d}}
\def\ddt{\frac{\rm d}{\dt}}

\newcommand{\dxdt}{\dx \,\dt}

\newcommand{\intOL}[1]{\int_{\Omega_L} #1 \ {\rm d} {y}}
\newcommand{\intOM}[1]{\int_{\Omega} #1 \ \dx}

\newcommand{\intTOL}[1]{\int_0^T\!\!\!\! \int_{\Omega_L} #1 \ {\rm d}{y}\,{\rm d}{\tau}}
\newcommand{\intTOM}[1]{\int_0^T\!\!\!\! \int_{\Omega} #1 \ \dxdt}

\begin{document}

\title*{Singular Cucker-Smale Dynamics}
\titlerunning{SCSD}
\author{Piotr Minakowski, Piotr B. Mucha, Jan Peszek, and Ewelina Zatorska}
\authorrunning{P. Minakowski, P.B. Mucha, J. Peszek, and E. Zatorska}
\institute{Piotr Minakowski \at Institute of Analysis and Numerics, Otto von Guericke University Magdeburg,  Universit\"atsplatz 2, 39106 Magdeburg, Germany \email{piotr.minakowski@ovgu.de}
\and Piotr B. Mucha \at Institute of Applied Mathematics and Mechanics, University of Warsaw,  ul. Banacha 2, 02-097 Warszawa, Poland \email{p.mucha@mimuw.edu.p}
\and Jan Peszek \at Center for Scientific Computation and Mathematical Modeling (CSCAMM),
University of Maryland, College Park, MD 20742-4015, USA.\newline
Institute of Mathematics of the Polish Academy of Sciences,
 ul. \'Sniadeckich 8
00-656 Warszawa, Poland.
\email{j.peszek@mimuw.edu.pl}
\and Ewelina Zatorska \at  Department of Mathematics, University College London,  Gower Street, London WC1E 6BT, United Kingdom \email{e.zatorska@ucl.ac.uk}
}
%
%

\maketitle
\abstract{
The existing state of the art for singular models of flocking is overviewed, starting from microscopic model of Cucker and Smale with singular communication weight, through its mesoscopic mean-filed limit, up to the corresponding macroscopic regime. For the microscopic Cucker-Smale (CS) model, the collision-avoidance phenomenon is discussed, also in the presence of bonding forces and the decentralized control. For the kinetic mean-field model, the existence of global-in-time measure-valued solutions, with a special emphasis on a weak atomic uniqueness of solutions is sketched. Ultimately, for the macroscopic singular model, the summary of the existence results for the Euler-type alignment system is provided, including existence of strong solutions on one-dimensional torus, and the extension of this result to higher dimensions upon restriction on the smallness of initial data.  Additionally, the  pressureless Navier-Stokes-type system corresponding to particular choice of alignment kernel is presented, and compared -- analytically and numerically -- to the porous medium equation. }
\section{Introduction}

The phenomenon of {\it flocking} can be understood as a general tendency of self-propelled particles (or agents) to organize their dynamics based on the behaviour of their neighbours. It is a process in which an ensemble of particles aligns their velocities,  remaining in a close proximity to each other, typical for herds of mammals, schools of fish, or flocks of birds. Surprisingly enough, due to the inherent flexibility of mathematical modelling, this basic phenomenon can be used to describe a variety of seemingly unrelated processes. Indeed, achieving consensus,  synchronisation of motion, or emergence of complex structures and patterns are observable in much more diverse areas like:  distribution of goods \cite{goods}, spacecraft formation \cite{perea}, sensor networks \cite{sensor} and digital media arts \cite{arts}, as well as emergence of languages in primitive societies \cite{lang}. For the survey on the multi-agent systems and their applications we refer the reader to \cite{massurvey}.

The Cucker-Smale (CS) flocking model,  introduced in the seminal work \cite{cuc1} of Cucker and Smale, is a basic example of a flocking model that concentrates on the alignment of particles' velocities. 
In physical language it means that the velocity vector of each 
individual is determined in terms of positions and momenta of other members of the group.
Perhaps the simplest example is a {\it flock of two birds} (see [Section IV in \cite{cuc1}]) with positions and velocities equal to $(x_i(t),v_i(t))$, for $i=1,2$. Their dynamics can be described as follows
%
%
\begin{equation}\label{2birds}
\ddt x=v, \qquad \ddt v = - \frac{ v  }{(1+|x|)^\alpha} \mbox{ \ \ with \ } \alpha >0,
\end{equation}
where 
\begin{equation*}
 x(t)=x_1(t)-x_2(t), \qquad v(t) = v_1(t) - v_2(t).
\end{equation*}
As a basic feature of the model one gets alignment of velocities, i.e.
\begin{equation*}
 v(t) \to 0 \mbox{  \ \ as \ \ } t \to \infty.
\end{equation*}
The complete information on possible dynamics of system \eqref{2birds} is encoded in the communication weight, which in this  case is equal to $(1+|x|)^{-\alpha}$. The communication weight represents the perception of the particles and, in general, dampens the long-range interactions between the particles. The same principle as in \eqref{2birds} applied to multiple agents leads to the general CS flocking model
\begin{subnumcases}{\label{cs}}
\ddt x_i=v_i,\label{cs1}\\
\ddt v_i=\displaystyle\frac{1}{N}	\displaystyle\sum_{j=1}^N(v_j-v_i)\psi(|x_i-x_j|),\label{cs2}
\end{subnumcases}
with the initial conditions
$$ (x_1,\ldots,x_N)(0)=\vx_0,\ (v_1,\ldots,v_N)(0)=\vv_0.$$ 
Here, $N$ denotes the number of the particles, $x_i(t)$ and $v_i(t)$ are the position and the velocity of $i$th particle at the time $t$, and $\psi$ is the aforementioned communication weight, which is usually assumed to be positive, non-increasing and smooth.

After introduction of the CS model in 2007, inspired to some degree by the work or Vicsek {\it et. al.} \cite{vic} form 1995, extensive studies were carried out in various directions including: asymptotics \cite{haliu, cuc1}, pattern formation \cite{park, perea}, collision avoidance \cite{ahn1, cuc2}, variants of the model with preferences \cite{cuc3, Li}, leadership \cite{C-D, shen} and additional deterministic or stochastic forces \cite{bocan, hahaki, car3, habel}. Further directions include the study of the kinetic \cite{hatad, can, haliu} and hydrodynamic \cite{K-M-T, H-K-K1, H-K-K2} limits of the CS particle system and their coupling with classical equations of hydrodynamics \cite{Bae, bae2, choi, mpp}.
Other interesting variants of the CS model revolve around changing the symmetric all-to-all character of the interactions. As examples let us mention interactions with cone-shaped sensitivity regions \cite{C-C-H-S}, weighted normalization \cite{mo} (known as the Motsch-Tadmor model), or topological interactions \cite{Hask, S-T}. For an exhaustive overview of the research on the CS model with a regular communication weight we refer the reader to \cite{B-D-T1}, and references therein. In the present survey we omit certain directions that were covered in \cite{B-D-T1} focusing solely on the case of singular $\psi$. We explain this direction of research below. 


Most of the analysis of system (\ref{cs}) has been done for the case of regular, bounded communication weight. The CS model with such weight can be treated as a nonlinear ODE system with a Lipschitz continuous nonlinearity. Such systems are relatively well understood, as far as their basic properties  like well-posedness, are concerned. Also from the application perspective, regular communication weight is often perfectly suitable. However, certain phenomena exhibiting strongly local interactions require the use of singular communication weights that blow up whenever any two particles collide. One of the most distinctive characteristics of the CS model with singular weight, that opens  further possibilities of application, is that the particles avoid collisions regardless of the initial data.
This chapter is dedicated to the description of the broad dynamics arising from the singular CS model. We  will consider system (\ref{cs}) not only as a law governing the motion of a number of particles, but we will also discuss the possible dynamics in different scales: 
meso- and macroscopic.  The first one is relevant for systems with  very large number of the particles, while the macroscopic (hydrodynamical) level becomes convenient when condensation of particles it too large to distinguish single evolution of any one of them.

The rest of the Chapter is organized as follows. In Section \ref{prelim}, we present the necessary preliminaries. Section \ref{part} is dedicated to the singular CS particle system and, particularly, to the collision-avoidance and its applications. In Section \ref{kinet}, we present the results concerning the kinetic CS equation. Finally, in Section \ref{hydro}, we discuss a particular version of hydrodynamic CS model, known as the fractional Euler alignment system, which we also compare to the porous medium equation. 

\section{Preliminaries}\label{prelim}

Before presenting the results directly related to the singular CS model, let us introduce  basic properties that can be derived from the structure of system \eqref{cs}. First, we define the four basic states of the dynamics.
\begin{defn}\label{basic}~
\begin{enumerate}
\item[i)] We say that $i$th and $j$th particles {\underline{collide}} at time $t_0$ iff
\begin{align*}
x_i(t_0) = x_j(t_0).
\end{align*}
\item[ii)] We say that $i$th and $j$th particles \underline{stick together} at time $t_0$ iff they collide at $t_0$ and
\begin{align*}
v_i(t_0) = v_j(t_0).
\end{align*}
\item[iii)] We say that the system \eqref{cs} \underline{aligns asymptotically} iff
\begin{align*}
\sum_{i,j=1}^N|v_i(t)-v_j(t)|^2\to 0,\qquad \mbox{as}\qquad t\to\infty.
\end{align*} 
\item[iv)] We say that the system \eqref{cs} \underline{flocks asymptotically} iff it aligns asymptotically and there exists a constant $D>0$ such that 
\begin{align*}
\limsup_{t\to\infty}\max_{i,j=1,...,N}|x_i(t)-x_j(t)|\leq D.
\end{align*} 
\end{enumerate}

\end{defn}

\begin{rem}\rm
Definition \ref{basic} is admissible only in the case of the Cucker-Smale particle system \eqref{cs}, while throughout the chapter we will also discuss the kinetic and hydrodynamic regimes. Note, unlike points (i) and (ii), points (iii) and (iv) can be easily generalized to the dynamics of continua. 
\end{rem}

Summing equation \eqref{cs2} with respect to $i=1,...,N$ we deduce that
\begin{align*}
\frac{d}{dt}\sum_{i=1}^Nv_i = \frac{1}{N}\sum_{i,j=1}^N(v_j-v_i)\psi(|x_i-x_j|) = 0,
\end{align*}
which means that the average velocity is constant, i.e.
$\displaystyle \frac{1}{N} \sum_{i=1}^Nv_i = const.$
Therefore, throughout the chapter,  we assume without a loss of generality that
\begin{align}\label{mean0}
\frac{1}{N} \sum_{i=1}^Nv_i = 0,\qquad \mbox{and}\qquad \frac{1}{N} \sum_{i=1}^Nx_i = 0.
\end{align}
This assumption implies equivalence between alignment and  dissipation of kinetic energy:
\eq{\label{Ek}
E_k:=\frac{1}{2}\sum_{i=1}^N v_i^2 = \frac{1}{4N}\sum_{i,j=1}^N|v_i-v_j|^2.
}

The following proposition combines two most important structural properties of the CS system: dissipation of the kinetic energy and boundedness of the velocity.

\begin{prop}\label{dis}
Let $(x,v)$ be a smooth solution to system \eqref{cs}. Then 
\begin{align}\label{dis0}
\ddt\sum_{i=1}^N v_i^2 = -\frac{1}{N}\sum_{i,j=1}^N(v_i-v_j)^2\psi(|x_i-x_j|)\leq 0.
\end{align}
In particular, the kinetic energy $E_k(t)$ is bounded by the initial kinetic energy $E_k(0)$ and on top of that $|v_i(t)|\leq \sqrt{2E_k(0)}$ for all $t\geq 0$ and all $i=1,...,N$.
\end{prop}

\begin{proof}
By \eqref{cs2}, we have
\eq{\label{trick}
\ddt \sum_{i=1}^N v_i^2 &= \frac{2}{N}\sum_{i,j=1}^N v_i\cdot (v_j-v_i)\psi(|x_i-x_j|)\\
&= \frac{1}{N}\sum_{i,j=1}^Nv_i\cdot (v_j-v_i)\psi(|x_i-x_j|) + \frac{1}{N}\sum_{i,j=1}^Nv_j\cdot (v_i-v_j)\psi(|x_i-x_j|)\\
&= -\frac{1}{N}\sum_{i,j=1}^N(v_i-v_j)^2\psi(|x_i-x_j|),
}
where equality \eqref{trick} is obtained by swapping the indexes $i$ and $j$ in the second term.\hfill$\square$
\end{proof}

\begin{rem}\rm
Proposition \ref{dis} is based only on the structure of the CS system, and holds for any non-negative communication weight $\psi$. The symmetry of $x\mapsto \psi(|x|)$ allows to swap indexes in \eqref{trick}.
\end{rem}

Throughout the chapter, we assume that the singular communication weight is of the form
\begin{align}\label{psing}
\psi(s) = s^{-\alpha},\qquad \alpha>0.
\end{align}
This form is convenient for distinction between two cases: the \emph{weakly singular} case corresponding to $\alpha\in(0,1)$, and the \emph{strongly singular} case corresponding to $\alpha\geq 1$.
However, in reality, such exact form is not necessary, and the results can be generalized to an arbitrary weight that is positive, locally Lipschitz continuous on $(0,\infty)$, and singular at $0$. In such case the weakly and strongly singular cases translate to integrability of the weight around $0$, or its lack, respectively.

\section{Singular Cucker-Smale model: particle system}\label{part}
\subsection{Motivation}
One of the most desirable qualitative features of the CS system \eqref{cs}, either for standard or for the singular weight,  is collision avoidance. Such property is required in the fields where the agents naturally avoid collisions, such as, behavior of flocks of animals or control over autonomous sensors or robots.
One of the approaches in the study of collision avoidance comes from \cite{ahn1} where the authors establish a set of initial data such that the regular CS particle system admits no collisions. Roughly speaking, the initial total kinetic energy has to be small compared to the initial minimal distance between the particles. Then the alignment force on the right-hand side of \eqref{dis0} dissipates the kinetic energy and the rate of the dissipation increases as the distance between particles becomes smaller.

This effect motivates consideration of the CS with the singular communication weight \eqref{psing}, for which one expects that the rate of dissipation of the kinetic energy  becomes infinite as the particles collide. It turns out to be a good approach as it lead to collision-avoidance that is {\it unconditional}, i.e. it does not rely on the initial configuration.
The CS particle system with weight \eqref{psing} provides an interesting mathematical challenge owing to the fact that its nonlinear right-hand side is not Lipschitz continuous. However, since the system looses its regularity only at times of collisions, careful qualitative analysis of the model provides information required for the quantitative analysis. The following subsections are dedicated to the summary of results in  two main areas of analysis of this model: simultaneous quali-quantitative analysis and asymptotics, based on papers \cite{ccmp, jpe, jps}. 

At the end of the section we also provide examples of the influence of the singular kernel when coupled with a pattern-inducing control.

\subsection{Collision-avoidance}\label{Sec:3.2}

 The dynamics of the singular CS model is defined by the right-hand side of \eqref{cs2} and particularly by the interplay between 
$(v_j-v_i)$ which tends to zero and $\psi(|x_i-x_j|)$ which tends to infinity at the time of collision between $i$th and $j$th particle. Interestingly, the results of such interplay vary dramatically depending on the values of exponent $\alpha$.
Suppose $N$ particles are governed by the CS model with a singular weight \eqref{psing} on the time interval $[0,T]$. Then, equalities \eqref{Ek} and \eqref{dis0} put together give
\begin{align*}
\frac{d}{dt}E_k = -\frac{2}{N}\sum_{i,j=1}^N(v_i-v_j)^2
\psi(|x_i-x_j|).
\end{align*}
It implies that the kinetic energy dissipates at the integrable rate, i.e.
\eq{\label{diss}
\frac{2}{N}\int_0^T\sum_{i,j=1}^N(v_i-v_j)^2\psi(|x_i-x_j|)\leq
E_k(0).
}
Therefore, if the function
\begin{align}\label{theta}
\theta_{ij}(t):=\psi(|x_i(t)-x_j(t)|)
\end{align}
integrates to infinity in a neighbourhood of $t_0\in[0,T]$ then, in order to ensure \eqref{diss}, we necessarily need 
\begin{align}\label{st}
(v_i(t)-v_j(t))^2\to 0,
\end{align}
for $t\to t_0$. This suggests that 
to analyze collisions one should first analyze maps $\theta_{ij}$. In the case of strongly singular kernel $\psi$ with $\alpha\geq 1$, assuming that $i$th and $j$th particles 
collide at $t_0$, we have for $s\nearrow t_0$
\begin{align*}
|x_i(s)-x_j(s)|\leq M(t_0-s),
\end{align*}
where, by Proposition \ref{dis}, $M:=\sqrt{2E_k(0)}$ is the uniform bound for the velocity. Therefore,
\begin{align*}
\theta_{ij}(s)=\psi(|x_i(s)-x_j(s)|)\geq \psi(M(t_0-s)) = M^{-\alpha}|t_0-s|^{-\alpha},
\end{align*}
which for $\alpha\geq 1$ is non-integrable in any neighbourhood of $t_0$, and thus indeed \eqref{st} is necessary. On the other hand, for $\alpha\in(0,1)$ the function $|t_0-s|^{-\alpha}$ is integrable and the above argumentation is inconclusive. 
It was shown in \cite{jpe} that with $\alpha\in(0,1)$ functions $\theta_{ij}$ can be either integrable or non-integrable.

We summarize the above consideration in the following remark.
\begin{rem}\label{dyn}\rm
Assuming that $t_0$ is a time of collision between $i$th and $j$th particles we learn the following:
\begin{enumerate}
\item For $\alpha\geq 1$, function  $\theta_{ij}$ is non-integrable at $t_0$. Consequently $(v_i-v_j)^2\to 0$. Thus if any particles collide, then they stick together at the same time. Later we show that collisions are actually impossible.
\item For $\alpha\in(0,1)$, function $\theta_{ij}$ may be either integrable or nonintegrable at $t_0$. If it is integrable we have a collision between particles and 
if it is nonintegrable then $(v_i-v_j)^2\to 0$ and  particles stick together.
\end{enumerate}
\end{rem}

With this information we present  main results concerning the quali-quantitative analysis of the singular CS model. First, let us focus on the case of strongly singular weight with $\alpha\geq 1$, which leads to collision-avoidance. 

\begin{theo}[\cite{ccmp}]\label{thm1}
Let $\alpha\geq 1$ and $T\in(0,\infty]$. Then the CS particle system with singular weight \eqref{psing} admits a unique non-collisional smooth solution $(\vx,\vv)$ provided that  initial data are non-collisional (we recall Definition \ref{basic}$(i)$ for the notion of a collision).
\end{theo}
{\it Idea of the proof.}
Local existence of solutions is clear as the system is singular only at times of collision, and we begin with non-collisional initial data. This solution can be prolonged until the first time of collision, which as we prove, never happens.
As observed in Remark \ref{dyn}, if the particles collide, their relative velocity tends to zero. However,  the rate of the alignment outweighs the speed with which the particles approach each other, and as a result, the collision never occurs. We present the proof of this fact for the simplest case of two-particle system in $\R$. 

Assume that $t<t_0<\infty$, where $t_0$ is a time of the first collision, and denote $x(t):=x_1(t)-x_2(t)$ and $v(t):=v_1(t)-v_2(t)$. Using  \eqref{dis0}, we get
\eqh{
\ddt v^2 = -2\psi(|x|)v^2,}
equivalently, dividing by $|v|$, we have
\eqh{
\ddt |v| = -\psi(|x|)|v|,
}
and so, by Gronwall's lemma
\begin{align*}
|v(s)|\leq e^{-\int_0^s\psi(|x(\sigma)|)d\sigma}|v_0|,
\end{align*}
for any $s<t_0$. Then the primitive function of $\psi$, denoted by $\Psi$, satisfies
\begin{align*}
|\Psi(|x(t)|)|\leq |\Psi(|x(0)|)| + |v_0|\int_0^{t}\psi(|x(s)|)e^{-\int_0^s\psi(|x(\sigma)|)d\sigma}\dd s\leq C.
\end{align*}
Since $\Psi$ is singular at $0$ (singularity of the order $s^{1-\alpha}$ for $\alpha>1$ and of the order $\ln s$ for $\alpha=1$), we conclude that there exists $\delta>0$ such that $|x(t)|\geq\delta$ as $t\nearrow t_0$ which contradicts the assumption that $t_0$ is a time of collision.

In \cite{ccmp} we expand on the above idea. We divide the particles into two groups: $A$ -- of all particles colliding with $i$th particle at the time $t_0$, and $B$ -- of all remaining particles. 
Then, in a neighbourhood of $t_0$ there exists a minimal distance between groups $A$ and $B$. Therefore the dynamics of group $A$ is influenced by the singular interaction within $A$ and a negligible in comparison, bounded interaction between the groups. Thus half of the singular interaction within $A$ outweighs the influence of $B$ and then the remaining half is used to prove that the collision cannot happen similarly to the case of two particles. \hfill$\square$

Our next goal is to discuss the case of $\alpha\in(0,1)$. With $\alpha\geq 1$ trajectories of particles cannot cross while, as we show later, the weakly singular kernel with $\alpha<1$ admits not only collisions between the particles but also sticking. This leads to two hypothetical problems.

The first problem is related to the uniqueness. Since $\psi$ is singular at $0$ the loss of uniqueness can happen at any time at which any two particles are stuck together.
For example, if $i$th and $j$th particle have the same position and velocity in the time interval $(t_1,t_2)$, then their trajectories may separate at any time $t>t_2$ as 
one observers in the case of the basic example $\dot{x}=x^\frac{1}{3}$. 
The second problem rises from the fact that upon approaching the first (or any other) time of sticking of particles, we lose the absolute continuity of the solution, and its derivative cannot be defined in the weak sense.

We deal with these two delicate hypothetical problems using an, admittedly, heavy-handed approach. In a neighbourhood of any time at which no particles are stuck together the problems do not exist. As we approach a time of sticking $t_0$ we say that the solution exists in a classical sense up to any $t<t_0$ and  its position and velocity components are continuous at $t_0^-$. Thus, even though the classical meaning of the solution is no longer available, we still can prolong it up to $t_0$. When  $t_0$ is reached, we redefine the system in order to remove the possibility of separation of the trajectories. Then, we re-initiate the solution starting from $t_0$ which, again, is classical up to the second time of sticking $t_1$. This procedure is repeated up to at most $N$ times, since this is the maximum number of sticking between the particles (after we ensured that the trajectories cannot separate). To employ this strategy we denote
\begin{align*}
B_i(t):=\{k=1,...,N:x_k(t)\neq x_i(t)\ \mbox{or}\ v_k(t)\neq v_i(t)\},
\end{align*}
which can be interpreted as the set of indexes of all the particles that are distinct from $i$th particle. Then we define the solution as follows.
\begin{defn}[Piecewise weak solutions]\label{piecdef}
Let $\{t_n: n=0,1, ...,K\}$ with $K\leq N$ be the set of all times when the particles stick together. For $n\geq -1$, on each interval $[t_n,t_{n+1}]$ (we assume that $t_{-1}=0$) we consider the problem
\begin{subnumcases}{\label{pieccs}}
\ddt x_i = v_i,\\
\ddt v_i = \displaystyle\frac{1}{N} \sum_{j\in B_i(t_n)}(v_j-v_i)\psi(|x_i-x_j|)
\end{subnumcases}

with the initial data $(\vx(t_n),\vv(t_n))$.

We say that $(\vx,\vv)$ is a \underline{piecewise weak} solution of \eqref{cs} in the time interval $[0,T)$ with the initial data $(\vx_0,\vv_0)$ if and only if the function $(\vx,\vv)(t)$ is continuous on $[0,T)$ and it solves \eqref{pieccs} on each interval $[t_n,t]$ for all $t\in[t_n,t_{n+1}]$ and $(\vx(t_{-1}),\vv(t_{-1}))=(\vx_0,\vv_0)$.
\end{defn}

\begin{theo}[\cite{jpe,jps}]\label{piecthm}
Let $\alpha\in(0,1)$ and $T\in(0,\infty]$. Then for any initial data $(\vx_0,\vv_0)\in\R^{2d}$, system \eqref{cs} with the singular communication weight \eqref{psing} admits a unique solution in the sense of Definition \ref{piecdef}.
\end{theo}
{\it Idea of the proof.}
The proof of existence can be found in \cite{jpe}, while the proof of uniqueness can be found in \cite{jps}. Existence outside of times of collision is straightforward. At any time of collision we prove that due to the relatively small singularity exponent $\alpha$ the function $\theta_{ij}$ (see \eqref{theta}) is integrable and thus $\vv$ is absolutely continuous, which grants existence also in points of collision. Existence in points of sticking is dealt with mostly by the definition of the solution itself. Continuity of the velocity at the times of sticking is the only remaining problem, which is resolved by a careful elementary analysis of the dynamics.\hfill$\square$

It turns out, see \cite{jps}, that after restricting the range of singularity to $\alpha\in(0,\frac{1}{2})$, one obtains existence of classical solutions which, by uniqueness, coincide with the piecewise-weak solutions.

\begin{theo}[\cite{jps}]\label{piecthm2}
Let $\alpha\in(0,\frac{1}{2})$ and $T\in(0,\infty]$. Given initial data $(\vx_0,\vv_0)\in\R^{2d}$, system \eqref{cs} with the singular communication weight \eqref{psing} admits a unique classical solution with absolutely continuous velocity component.
\end{theo}
The advantage obtained by assuming  $\alpha\in(0,\frac{1}{2})$ is that the communication weight becomes square-integrable. In the proof we estimate the right-hand side of \eqref{cs2} using Young's inequality with exponent $2$, which leads to doubling exponent $\alpha$.

Introduction of the piecewise-weak solutions and the whole approach to the quantitative analysis of the CS model with singularity $\alpha\in(0,1)$ is based on the problems appearing at the times of sticking between particles. Therefore, a natural question arises if such phenomenon can even occur. In \cite{jpe} a detailed analysis of the two-particle in $\R^d$ case was performed which we sketch below.

Assuming without a loss of generality that $x_1 + x_2\equiv 0$ and $v_1+v_2\equiv 0$ (see \eqref{mean0}) we end up with two particles that move either on two parallel lines or on the same line. We omit the first possibility since it naturally leads to no collisions and focus on the situation which is equivalent to two particles in $\R$. Assume that $x_2(0)>x_1(0)$ and to make the particles move in the direction of each other we are forced to assume $v_2(0)-v_1(0)< 0$. Denoting $x:=x_2-x_1>0$ we use \eqref{cs} to obtain
\begin{align}\label{stick}
\ddot{x} = -\dot{x}\psi(x)
\end{align}
in $[0,t_0]$, where $t_0$ is the first time of collision between the particles.
\begin{prop}[\cite{jpe}]
Let $\alpha\in(0,1)$ and let $x$ be the unique solution\footnote{Which exists by Theorem \ref{piecthm}.} of \eqref{stick} with the initial data $x(0)>0$ and $\dot{x}(0)< 0$. Then the following are equivalent:
\begin{enumerate}
\item There exists a time $0<t_0<\infty$ such that $x(t_0)=\dot{x}(t_0) = 0$.
\item We have
\begin{align}\label{stickini}
\dot{x}(0)=-\Psi(x(0)),
\end{align}
where $\Psi(s) :=  \frac{1}{1-\alpha}s^{1-\alpha}$ is a primitive of $\psi$.
\end{enumerate} 
\end{prop}

The proof can be found in \cite{jpe}. Here we shall only make the observation that integrating \eqref{stick} in the time interval $[0,t]$ leads to
\begin{align*}
\dot{x}(t) + \Psi(x(t)) = \dot{x}(0) + \Psi(x(0)) \stackrel{\eqref{stickini}}{=} 0.
\end{align*}
It means that the initial condition \eqref{stickini} places the solution on the trajectory described by $\dot{x} = -\Psi(x)$. For such trajectory if $x=0$ then $\Psi(x)=0$ and thus $\dot{x}=0$, which implies that the particles stick together whenever they collide. Then the proof of the proposition revolves around showing that this is the only trajectory leading to a collision and that the collision happens in a finite time.

\subsection{Asymptotics}

The asymptotics of the singular CS model is mostly the same as the asymptotics of the regular one, since it is related to the integrability of $\psi$ away from zero. This case was thoroughly studied on particle, kinetic and hydrodynamic levels by Ha and Liu in \cite{haliu}, Carrillo {\it et. al.} in \cite{car}, Ha and Tadmor in \cite{hatad} and others. We also recommend the survey \cite{B-D-T1}, where the regular CS model was discussed. Although the asymptotic behaviour of solutions for singular kernels is not significantly different, we discuss it here for the sake of completeness using the results from \cite{haliu} as an example. 
Since the integrability of $\psi(s)$ at $s=\infty$ is a major factor, we again distinguish based on $\alpha$.
\begin{prop}[Unconditional flocking, \cite{haliu}]\label{asymppart}
Let $\alpha\in(0,1]$ and let $(x,v)$ be a solution to \eqref{cs} with $|\vx_0|\neq 0$. Then there exist positive constants $x_m$ and $x_M$ such that
\begin{align*}
x_m\leq \sum_{i,j=1}^N(x_i-x_j)^2\leq x_M,\qquad \|\vv\|:=\sum_{i,j=1}^N(v_i-v_j)^2 \leq \|\vv_0\|e^{-\psi(x_M)t}.
\end{align*}
\end{prop}

\begin{prop}[Conditional flocking, \cite{haliu}]\label{asympart2}
Let $\alpha>1$ and let $(\vx,\vv)$ be a solution to \eqref{cs} with $|\vx_0|\neq 0$. Suppose the initial configuration $(\vx_0,\vv_0)$ satisfies
\begin{align*}
\left(\sqrt{\sum\nolimits_{i,j=1}^N(x_{0i}-x_{0j})^2}\right)^{1-\alpha}\geq (\alpha-1)\sqrt{\sum\nolimits_{i,j=1}^N(v_{0i}-v_{0j})^2}.
\end{align*}
Then there exist positive constants $x_m$ and $x_M$ such that
\begin{align*}
x_m\leq \sum\nolimits_{i,j=1}^N(x_i-x_j)^2\leq x_M,\qquad \|\vv\|:=\sum\nolimits_{i,j=1}^N(v_i-v_j)^2 \leq \|\vv_0\|e^{-\psi(x_M)t}.
\end{align*}
\end{prop}

\begin{rem}\rm
The class of admissible solutions required by Propositions \ref{asymppart} and \ref{asympart2} includes the classes provided by Theorems \ref{piecthm}  and \ref{piecthm2} for $\alpha<1$ and by Theorem \ref{thm1} for $\alpha\geq 1$.
Observe that singularity with $\alpha=1$ is the only value that satisfies the assumptions of Theorem \ref{thm1} and Proposition \ref{asymppart} and thus, leads to, both, unconditional flocking and collision-avoidance. Of course, as explained below equation \eqref{psing}, the choice of the communication weight of the form \eqref{psing} is quite arbitrary and, in practice, the only requirement for the lack of collisions is the nonintegrability of $\psi$ near $0$. On the other hand for the unconditional flocking nonintegrability of $\psi$ at the infinity is required.
\end{rem}
\subsection{Variants of the model}

From the perspective of applications, it is often useful to modify the Cucker-Smale model to adapt it to particular phenomena. We recall the wide range of modifications, presented in the introduction, from models with leaders \cite{shen, cuc3, C-D} and preferences \cite{cuc3, Li}, models with time-delay \cite{has}, up to models with various additional external or internal forces (deterministic and stochastic) \cite{hahaki, bocan, park, car3}.
We also refer to the survey in \cite{B-D-T1}. However, in the case of the CS model with a singular communication weight the well-posedness theory is relatively fresh and thus not many additional directions were pursued as of yet. Moreover, as presented in previous sections, the dynamics of the singular CS model either admits sticking of the trajectories of the particles or does not allow any collisions at all. In the first case, admittedly, not many perspectives of applications were discovered and in the second case, the dynamics is essentially equivalent to the regular CS model with an added bonus of initial-data-independent collision-avoidance. In particular, singular CS model with $\alpha\geq 1$ seems to be viable for most modifications that the regular CS model underwent with some additional mathematical challenge. That being said, in the remainder of this section we present two results directly involved with the singular CS model.

{\bf Bonding force.} The first variant comes from paper \cite{K-P} and it deals with the CS model with a bonding force. The bonding force was introduced for the regular CS model by Park {\it et. al.} in \cite{park}. The system reads

\begin{subnumcases}{\label{B-2}}
\frac{\dd x_i}{\dt} = v_i,\quad x_i,v_i\in\R^d,\quad i=1,2,\cdots,N,\quad t>0\label{B-2a}\\
\begin{aligned}
\frac{\dd v_i}{\dt} =& \frac{K_1}{N}\sum_{j=1}^N\psi(|x_j-x_i|) (v_j-v_i)
+\frac{\tilde{K}}{N}\sum_{j=1}^N \frac{(v_i-v_j) \cdot (x_i-x_j)}{2|x_i-x_j|^2}(x_j-x_i)\\
&+\frac{K_2}{N} \sum_{j=1}^N\frac{|x_j-x_i|-2R}{2|x_j-x_i|}(x_j-x_i),
\end{aligned}\label{B-2b}
\end{subnumcases}
where the middle and last terms on the right-hand side of \eqref{B-2b} compose the bonding force. Here constants $K_1$, $K_2$ and $\tilde{K}$ control the intensity of the interaction and constant $R$ influences the asymptotic distance between the particles.

The purpose of the bonding force in \eqref{B-2} is to impose a tendency for the particles to stay at distance $2R$ from each other.  However, for dimension related reasons such pattern is impossible for $N>d+1$. Instead, the numerical simulations performed in \cite{park} indicate that the particles converge to one of many configurations that possess the following properties: they are symmetric (see Figures \ref{fig:1c} and \ref{fig:1d}), the particles are contained within a ball of radius $2R$ and the distances between the particles are bigger than a positive constant. It is noteworthy that the latter two properties, while observed in simulations, were not proven mathematically in \cite{park}.

\begin{figure}[!tbp]
  \centering
  \subfloat[$N=20$]{\includegraphics[width=0.4\textwidth]{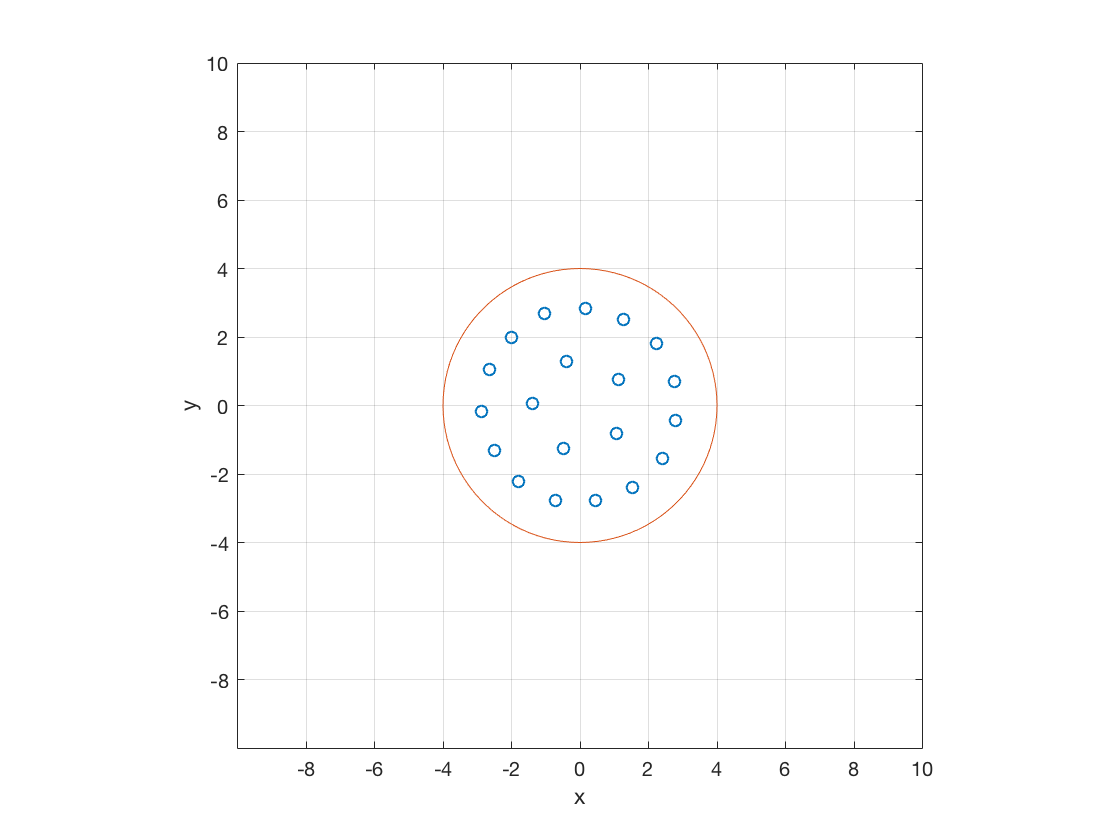}\label{fig:1c}}
  \subfloat[$N=25$]{\includegraphics[width=0.4\textwidth]{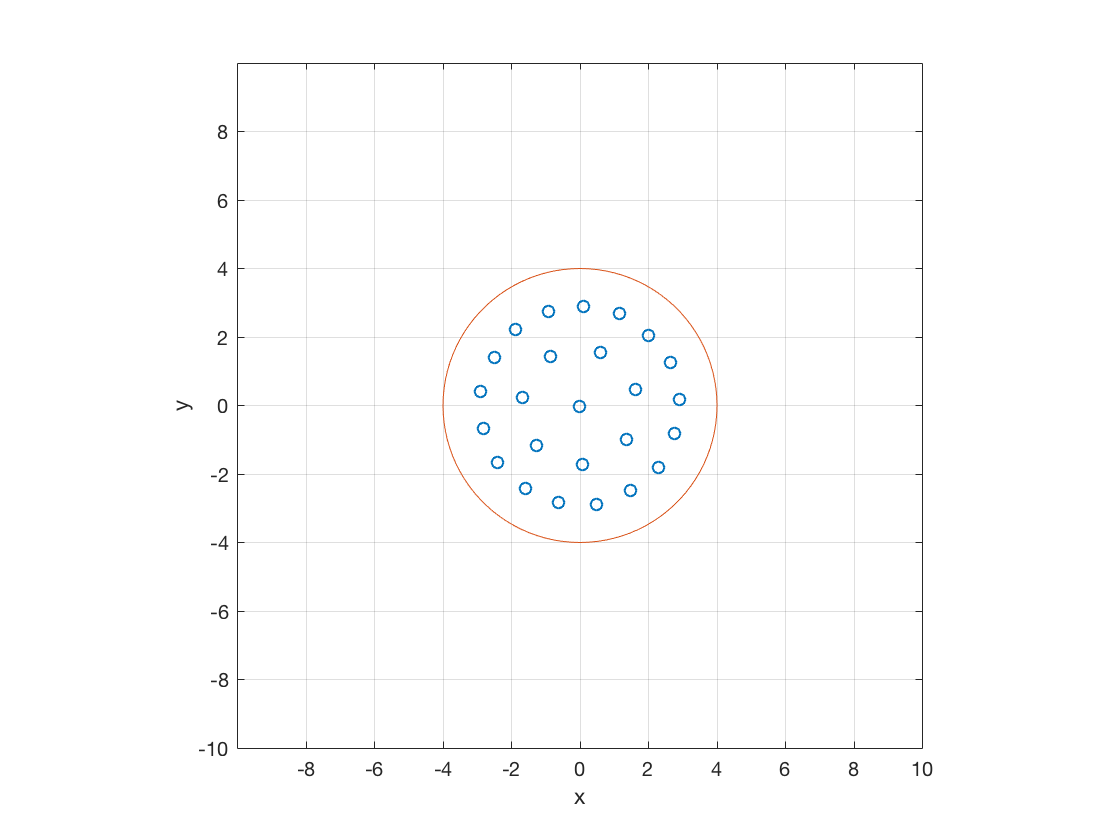}\label{fig:1d}}
  \caption{The particles governed by system \eqref{B-2} coverage to a symmetric pattern that depends not only on the number of the particles (as seen on the pictures) but also on the initial configuration. Blue circles represent the particles, while the large red circle represents the $2R$ radius around the origin.}
\end{figure}

The contribution of \cite{K-P} is two-fold. 

First, through the implementation of a singular weight, collision-avoidance is shown. Second, it is proven that asymptotically, the particles converge to a state at which they are contained in a ball of radius $2R$, and the distances between them are positive, thus the gap between the numerical observations and theoretical knowledge from \cite{park} is bridged. We note that, furthermore, the lack of local and asymptotic collisions leads to a global-in-time minimal distance between the particles. These contributions are summarized  in the following theorem.

\begin{theo}[\cite{K-P}]\label{k-p}
Consider system \eqref{B-2} in two frameworks:
\begin{description}
\item[${\mathcal F}_1$] with regular communication weight, e.g. $\psi(s)=(1+s)^{-\alpha}$,
\item[${\mathcal F}_2$] with singular communication weight \eqref{psing} with $\alpha\geq 1$ and non-collisional initial data. 
\end{description}
Then both systems in both frameworks admit unique smooth solutions such that
\begin{enumerate}[(i)]
\item the kinetic energy $E_k:=\frac{1}{2}\sum_{i=1}^N|v_i|^2$ converges asymptotically to zero,
\item there exists $\rho>0$ such that
\begin{align*}
\liminf_{t\to\infty}\min_{i,j=1,...,N}|x_i(t)-x_j(t)|\geq \rho,
\end{align*}
\item we have
\begin{align*}
\limsup_{t\to\infty}\max_{i,j=1,...,N}|x_i(t)-x_j(t)|\leq 2R.
\end{align*}
\end{enumerate}
Moreover in framework ${\mathcal F}_2$ point $(ii)$ can be replaced with
\begin{enumerate}[(i)]
\setcounter{enumi}{1}
\item\hskip-.2 em $^{'}$  there exists $\rho>0$ such that
\begin{align*}
\inf_{t\geq 0}\min_{i,j=1,...,N}|x_i(t)-x_j(t)|\geq \rho.
\end{align*}
\end{enumerate}
\end{theo}
The proof of existence and point $(i)$ in framework ${\mathcal F}_1$ can be found in \cite{park}. All other assertions are shown in \cite{K-P}. Here we briefly describe the reasoning in the case of framework ${\mathcal F}_2$.

Proof of existence follows from the collision-avoidance exactly like in Theorem \ref{thm1}. In the proof of Theorem \ref{thm1} collision-avoidance is shown by dividing the interactions between the particles into two groups: $A$ -- singular interactions, and $B$ -- bounded interactions. Here all the interactions originating from the bonding force are added to the group $B$ (since bonding force is not singular) and then the proof proceeds in the same way as in Theorem \ref{thm1}. 
For a  proof of $(i)$ we refer to  \cite{park}. As for point $(ii)$, a careful analysis reveals that an asymptotic collision between particles is an unstable event and makes convergence of the kinetic energy impossible (and thus, contradicts $(i)$). The proof of $(iii)$ follows thanks to the uniform-in-time regularity of the solution granted by point $(ii)$.

\begin{rem}\rm
In \cite{K-P} Theorem \ref{k-p} is proven in the case of a simplified version of system \eqref{B-2}, where the middle term (with $\tilde{K}$) in \eqref{B-2}$_2$ is removed. Such simplified system has the same asymptotics and collision-avoidance as \eqref{B-2}. Numerical simulations performed in \cite{K-P} revealed that without the middle term, the system converges to equilibrium at a slower rate, which suggests that this term is related to the mitigation of oscillatory behavior of the particles.
\end{rem}

{\bf Decentralized control.} Another modification of the singular CS model is through the addition of a decentralized control \cite{Bakule, J-L-M}. The idea is to assign a provisional order to the agents and make each agent (except for the first one) synchronize its position with the position of the previous one. The system reads
\begin{align}\label{cssparse}
\begin{aligned}
\frac{d x_i}{dt} &= v_i,\quad i=1,\dots, N, \quad t > 0,\cr
\frac{d v_i}{dt} &= \frac{K}{N}\sum_{j=1}^N \psi(|x_i-x_j|)(v_j - v_i) + u_i.
\end{aligned}
\end{align}
Given $z_i \in \R^d$ for $i=1,\cdots, N-1$, the control term $u := (u_1, \cdots, u_N)$ is given by
$$\begin{aligned}
u_1 &= -\phi(|x_1 - x_2 - z_1|^2)(x_1 - x_2 - z_1),\cr
u_N &= \phi(|x_{N-1} - x_N - z_{N-1}|^2)(x_{N-1} - x_N - z_{N-1}),\cr
u_i &= \phi(|x_{i-1} - x_i - z_{i-1}|^2)(x_{i-1} - x_i - z_{i-1}) - \phi(|x_i - x_{i+1} - z_i|^2)(x_i - x_{i+1} - z_i),
\end{aligned}$$
for $i \in \{2,\cdots, N-1\}$,
where $\phi$ is a smooth weight of the form $\phi(s)=(1+s)^{-\beta}$, $\beta>0$. Thus, through this control, in theory, $i$th particle adjusts its position in a way that minimizes $u_i$, which is by having $x_{i-1}-x_i$ converge to $z_{i-1}$. Therefore, by prescribing proper coordinates $z_i$, one can force $i$th particle to attain any position respective to $(i-1)$th particle.

Similarly to the model with a bonding force the main area of application is in the control of unmanned aerial vehicles. The advantage of the decentralized control is that it allows emergence of a variety of pattern formations through the manipulation of $z_i$. Moreover each agent is required to ``remember'' only its relative position to a single other agent. The disadvantage is that it requires input of $z_i$, while the bonding force achieves symmetric patterns depicted in Figures \ref{fig:1c} and \ref{fig:1d} automatically.

In \cite{ckpp} decentralized control was added to the singular CS model resulting in pattern formation with collision-avoidance presented in the following theorem.
\begin{theo}[\cite{ckpp}]\label{sparsethm}
(A) Consider system \eqref{cssparse} with $\alpha\geq 1$ and $\beta>0$ subjected to non-collisional initial data $(\vx_0, \vv_0)$ (see Definition \ref{basic}). Then there exists a global smooth, non-collisional solution.

(B) Moreover if $\alpha\geq 2$ and  one of the two following hypotheses holds:
\begin{itemize}
\item[(i)] $\beta \leq 1$; 
\item[(ii)] $\beta > 1$ and 
\begin{align*}
\sum_{i=2}^{N}\int_{|x_{0i-1} - x_{0i} - z_{i-1}|^2}^\infty \phi(r)\,\dd r > \frac{4}{MN}\sqrt{\sum_{i,j=1}^N|v_{0i} - v_{0j}|^2}.
\end{align*}
\end{itemize}
Then we have
\[
\sup_{0 \leq t \leq \infty}\max_{1 \leq i,j \leq N}|x_i(t)-x_j(t)| < \infty \quad \mbox{and} \quad \max_{1 \leq i,j \leq N}|v_i(t) - v_j(t)|\to 0 \quad \mbox{as} \quad t \to \infty.
\]

(C) Finally if
\begin{align}\label{noas}
\liminf_{t\to\infty}|x_i(t)-x_j(t)|>0
\end{align}
for all $i,j \in \{1,\dots,N\}$. Then there exists a limit $\lim_{t\to\infty}x(t)=:x^\infty$ satisfying
\begin{align*}
x_i^\infty=x_{i-1}^\infty-z_{i-1}\qquad\mbox{for all} \quad i=1,\dots,N.
\end{align*}
\end{theo}

\begin{rem}\rm
It is worthwhile to have a closer look at part (C) of the above theorem and particularly assumption \eqref{noas}. The reason to exclude the asymptotic collisions is related to collision avoidance. Take, for example, two particles in $\R$, with $z_1=-1$ then the resulting pattern has to be of the form $x_2^\infty = x_1^\infty+1>x_1^\infty$. However if initially $x_1(0)>x_2(0)$, then the particles change order, which means that they collide. This is however impossible by Theorem \ref{sparsethm} (A). In other words we need to exclude the situations when the control leads to a finite-time collision. Of course in $d\geq 2$ such situation is very unlikely.
\end{rem}

\section{Singular Cucker-Smale model: kinetic equation}\label{kinet}

\subsection{Formal derivation} 
The particle model provides the most precise description of the evolution of the particles, but in the case of large $N$, it quickly becomes impractical. With $N\to\infty$, the microscopic models are too computationally intensive, and it is much more efficient to perform numerical simulations for what we call {\emph{mean-field limit system}}, see for example \cite{BCDPZ}. This can be viewed as a model in the {\it mesoscopic} scale. Roughly speaking, instead of tracing position and momentum of each particle, we
look for the distribution (or 
probability) of the particles at time $t$, position $x$, and  with velocity $v$. Hence, the sough object of analysis is a distribution of the type
\begin{equation}\label{f-def}
 f(t,x,v): [0,T) \times \R^d_{x} \times  \R^d_{v} \to \R.
\end{equation}
At the right-hand side of the above expression, $\R$ should be viewed only formally, since $f$ might be barely a measure. The evolution of $f$ is described by the Vlasov-type system, which, along with Boltzmann equation, is the backbone of kinetic theory \cite{wasser}. The methodology developed to deal with Vlasov-type equations is robust but in the case of systems with singular interactions there is no general approach.
 
 Before we deliberate further on the matter, let us explain the link between the particle and the kinetic models. Assume that we have an $N$ particle system in  the following form
\begin{equation}\label{cs-m}
\left\{
 \begin{array}{l}
 \displaystyle  \ddt x_i=v_i,\\
 \displaystyle  \ddt v_i= \sum_{j=1}^N m_j (v_j -v_i) \phi(|x_i - x_j|),
 \end{array}\right.
\end{equation}
where  total mass of the particles, reads
$\sum_{j=1}^N m_j=1$
and for simplicity equals one. Note that \eqref{cs-m} becomes \eqref{cs} with $m_i=\frac{1}{N}$ for all $i=1,...,N$. Here, for the sake of clarity of presentation, we skip the dependence of $x_i=x_i^N$, $v_i=v_i^N$ and $m_i=m_i^N$ on $N$ but it should be noted that, naturally, the solution itself changes with $N$. 

Next, having a solution to (\ref{cs-m}), whose existence was discussed in the previous section, we aim to let $N\to\infty$ and define \eqref{f-def} as a limit of solutions to the particle system \eqref{cs-m} written as follows
\begin{equation}\label{atomic}
 f_N(t,x,v) = \sum_{i=1}^N m_i\delta_{x_i(t)} \otimes \delta_{v_i(t)},
\end{equation}
where $x_i$ and $v_i$ denote a position and velocity of $i$th particle obtained by solving the CS particle system \eqref{cs}. We will refer to \eqref{atomic} as the \emph{atomic solution}.

This way defined $f_N$ satisfies
\begin{equation}\label{x2}
 \ddt f_N=0, \mbox{ \ \ with  \ } \ddt=\partial_t + \frac{\dx}{\dt}\nabla_x + \frac{\dd v}{\dt} \nabla_v
\end{equation}
in the following distributional sense: for any
smooth test function 
\begin{equation*}
 \Phi: [0,T)\times \R^d_x \times \R_v^d \to \R
\end{equation*}
with a compact support in $\R^d_x\times \R^d_v$ and $\Phi|_{t=T} \equiv 0$, we have
\begin{equation*}
 \int_0^T \int_{\R_x^d} \int_{\R_v^d}  f_N(t,x,v) \ddt \Phi(t,x,v) \dd v \dxdt =\int_{\R^d_x}\int_{\R^d_v}f_N|_{t=0} \Phi(0,x,v) \dx\dd v.
\end{equation*}
Then by the definition of $f_N$ (and particularly \eqref{cs}) it is easy to show that
\begin{align}\label{dform1}
\int_0^T\int_{\R^d_x}\int_{\R^d_v}f_N[\partial_t + v\cdot\nabla_x + F(f_N)\cdot\nabla_v]\Phi(t,x,v)\dd v\dxdt = \nonumber\\
=\int_{\R^d_x}\int_{\R^d_v}f_N|_{t=0} \Phi(0,x,v) \dx\dd v,
\end{align}
where
\begin{align*}
F(f_N)(t,x,v)
& =  \sum_{j=1}^N m_j(t) (v_j-v) \psi(|x-x_j|)\\
&=\int_{\R^d_x}\int_{\R^d_v}(w-v)\psi(|x-y|)f_N(t,y,w)\dd y\dd w.
\end{align*}
Then, passing to the limit in \eqref{dform1} as $N\to \infty$, assuming  that each term is well defined and smooth, and that
\begin{equation}\label{x1}
 \lim_{N\to \infty}f_N =f,
\end{equation}
we deduce that $f$ satisfies the same equation as $f^N$.
It is a distributional version of the Vlasov-type equation
\begin{equation}\label{cs-k}
 f_t + v \cdot \nabla_x f + \div_v(F(f)f)=0
\end{equation}
with
\begin{equation}\label{cs-f}
 F(f)= \int_{\R^d_y} \int_{\R^d_w} (w-v) \psi(|x-y|) f(t,y,w) \dd y\dd w.
\end{equation}
Note that a very convenient property of Vlasov-type equations is that due to the {\emph{ nonlocal}} interactions (contrary to Boltzmann equation) the solution of the particle system already is a distributional solution to a Vlasov-type equation. This is the reason why such a simple approximation is possible.

%

In general, the limit passage $N\to\infty$ requires some more information about  uniform estimates for $f_N$. This is in fact the gist of the problem, and we postpone the discussion on this issue to the following sections. Here let us only briefly mention  that
even defining $F(f)f$, is not straightforward since  for the singular weight $\psi$, if $f$ is a Radon measure then $F(f)$ is an $L^p$ function, and products of  $L^p$ functions with measures might not be well defined.

As for the topology of convergence, for the method shown above, a suitable choice is the Wasserstein distance. For the sake of this survey we shall introduce a simple version of such metrics i.e. the bounded-Lipschitz distance. Given two Radon measures $f_1$ and $f_2$ let
\begin{equation}\label{bLd}
d_1(f_1,f_2)=\sup\left\{
\int_{\R^d} (f_1-f_2)\phi dx : \phi \in Lip(\R^d), Lip(\phi) \leq 1, |\phi|_\infty\leq 1\right\},
\end{equation}
where $Lip(\R^d)$ is the space of Lipschitz continuous functions and $Lip(\phi)$ is the Lipschitz constant of function $\phi$. Then $d_1$ is the bounded-Lipschitz distance, sometimes referred to as Monge-Kantorovich-Rubinstein distance and it is equivalent to the Wasserstein-1 distance,
we refer the reader to \cite{wasser} or \cite{vil}. By ${\mathcal M}$ we denote the metric space of all nonnegative Radon measures with topology generated by $d_1$. Its subspace of probabilistic measures is denoted by ${\mathcal P_1}$

\subsection{Local-in-time well-posedness}

A possible  approach to this subject, including singular communication weight, has been presented in \cite{carchoha}. To the best of our knowledge, it was the first existence result for the singular CS kinetic equation. Here we present only a special case of this result tailored to the singular CS model. Interested reader is refereed to  \cite{carchoha}, where a more general variant with nonlinear velocity coupling was analysed.
\begin{theo}[\cite{carchoha}]\label{ccm}
 Suppose $\alpha \in (0,d-1)$ and $p>1$ fulfils
$ (\alpha +1)p' < d$ with $p': \frac{1}{p}+\frac{1}{p'}=1.$
If initial datum $f_0$ is nonnegative and has a compact support in the velocity space and 
\begin{equation*}
 f_0 \in (L^1\cap L^p)(\R^d \times \R^d) \cap \mathcal{P}_1(\R^d\times \R^d),
\end{equation*}
then there exists $T>0$ such that there exists a unique weak solution
\begin{equation*}
0\leq f\in L^\infty(0,T;(L^1\cap L^p)(\R^d \times \R^d))\cap {\mathcal C}([0,T],  \mathcal{P}_1(\R^d\times \R^d))
\end{equation*}
for  system (\ref{cs-k}) on time interval $[0,T]$.

Furthermore, if $f_i$ with $i=1,2$ are two such solutions, then the following $d_1$-stability estimate holds
\begin{equation*}
 \ddt d_1(f_1(t),f_2(t)) \leq C d_1(f_1(t),f_2(t)) \mbox{ \ \ for \ \ } t \in [0,T]
\end{equation*}
for a positive constant $C$.
\end{theo}
The key point of Theorem \ref{ccm}
is the uniqueness. It follows from the fact that the singularity allows to consider a functional setting such that the field $F(f)$ is indeed Lipschitz continuous. 
The proof is based on the theory of optimal transport 
 to control two solutions considered in setting of the flow generated by $v$ and $F(f)$. Since, as we mentioned,
the regularity is reasonably high, there is no problem to define characteristics. Hence comparison of two solutions in the Wasserstein metric is possible. 

\subsection{Global-in-time measure-valued solutions}
The local existence result presented in the previous section requires  the initial data to belong to the $L^p(\R^d \times \R^d)$ space. In consequence, the solution itself is also an $L^p$-function, which rules out 
a very interesting class of solutions -- atomic solutions-- given by (\ref{atomic}).

The following result from \cite{mp} embraces the rich dynamics of the CS model admitting solutions that for all $t>0$ live in the space of Radon measures, which we  denote here by ${\mathcal M}$. Clearly, such class includes the atomic solutions. The price that needs to be paid for such a wide class of solutions is reduction of the range of singularity to $\alpha\in(0,\frac{1}{2})$, so that we can operate within the framework of higher regularity for the particle system granted by Theorem \ref{piecthm2}.

\begin{theo}[\cite{mp}]\label{main3}
Let $0<\alpha<\frac{1}{2}$. For any compactly supported initial data $0\leq f_0\in {\mathcal M}$ and any $T>0$, Cucker-Smale's flocking model (\ref{cs-k}) admits at least one  weak solution $0\leq f\in {\mathcal C}_{weak}([0,T], {\mathcal M}(\R^d\times \R^d))$ with {$\partial_t  f \in L^p (0, T ; 
(C^1(\R^d \times \R^d) )^* )$} for some $p>1$ (here $(C^1)^*$ is the dual space of $C^1$). 

Moreover if $f_0$ is atomic  then $f$ is atomic too, hence, by Theorem \ref{piecthm2}, it is unique and it corresponds to the solution of the particle system \eqref{cs}.
\end{theo}

{\it Idea of the proof.}
The proof of existence is based on an idea of mean field limit,  presented at the beginning of this section. It  involves
\begin{itemize}
\item[i)] Given initial data $f_0\in {\mathcal M}$, we approximate it by atomic measures of the form
\begin{equation*}
 f_{N,0}(t,x,v) = \sum_{i=1}^N m^N_{0i} \delta_{x^N_{0i}} \otimes \delta_{v^N_{0i}}.
\end{equation*}
\item[ii)] For each fixed $N$, thanks to the results from Section \ref{Sec:3.2} we establish existence of solutions to particle system \eqref{cs-m} with initial data $(\vx^N_0, \vv^N_0)$. By \eqref{atomic}, such solutions correspond to atomic solutions $f_N$ of Vlasov-type equation \eqref{cs-k} (satisfied at least in the distributional sense).
\item[iii)] We converge with $N\to\infty$ and define $f$ -- the candidate for the solution of \eqref{cs-k} associated with initial data $f_0$, as a limit of $f_N$ in the bounded-Lipschitz distance.
\item[iv)] Through sufficient information on the uniform regularity of the approximate solutions we prove that $f$ satisfies \eqref{cs-k}.
\end{itemize}

Points (i) and (ii) of the above scheme are clear, moreover,
the existence of a weak* limit $f$ is straightforward by Banach-Alaoglu theorem. The difficult part is to prove that $f$ is the sought solution to \eqref{cs-k}.
The crucial element of the proof is to show the convergence of the nonlinear alignment force term (compare with \eqref{dform1})
\begin{align}\label{MR1}
\int_0^T\int_{\R^{2d}}F(f_N)f_N\nabla_v\Phi \dx\dd v\dt,\quad\mbox{as}\quad N\to\infty.
\end{align}
In particular if $f_N\stackrel{*}{\rightharpoonup} f$ then it is not clear whether $F(f_N)f_N\stackrel{*}{\rightharpoonup} F(f)f$. It is useful to look at (\ref{MR1}) as
\begin{align*}
\int_0^T\int_{\R^{4d}}g d\mu_ndt\quad &\mbox{for}\ d\mu_N:= f_N(x,v,t)\otimes f_N(y,w,t)\dx\dd v\dd y\dd w\\
&\mbox{and}\ g(x,y,v,w,t):=\psi(|x-y|)(w-v)\nabla_v\Phi(t,x,v).
\end{align*}

Note that thanks to the above representation convergence of a product in \eqref{MR1} is reduced to a convergence of the measure $\mu_N$ tested by the function $g$. To overcome the fact that $g$ is not Lipschitz continuous we approximate it by a Lipschitz continuous family $g_m\to g$, such that $|g_m|\nearrow |g|$. First we converge with $N\to\infty$ and then with $m\to\infty$. Then a detailed analysis, based on the uniform regularity of trajectories provided by Theorem \ref{piecthm2}, leads to the proof that $F(f)$ is an integrable function  with respect to measure $df = f dx dv$.
Hence, we are able to define $F(f)f$, even though we cannot do it for $F(f)h$ for general $h\in {\mathcal M}$.\hfill$\square$

Theorem \ref{main3} provides existence of weak measured valued solutions globally in time, but does not solve the problem of uniqueness. In fact, the regularity is insufficient
to estimate the difference between the two supposedly distinct solutions. However, for a special case of atomic solutions we are able to 
obtain the so-called \emph{weak-atomic uniqueness} result. I states that if initial data is atomic, then the solution is atomic as well, and thus, by Theorem \eqref{piecthm2}, it is unique. The proof is based on the analysis
of possible supports of constructed solutions. Its main steps are presented in the following section.

\subsection{Weak-atomic uniqueness}
We present the idea of weak-atomic uniqueness: we aim to prove that if  the initial data is atomic, then the solution must be atomic too. 
%
%

It is sufficient to concentrate our attention on small times near $t=0$.
Let $f_0$ be atomic (recall \eqref{atomic}). Our goal is to restrict $f_0$ to small balls with just one particle 
(say $\underline{i}$th particle). Then we use the local propagation of the support to prove that 
the measure that initially formed the $\underline{i}$th particle remains atomic for short time. Since
\begin{align*}
f_0=\sum_{i=1}^Nm_i\delta_{x_{0i}}\otimes\delta_{v_{0i}},
\end{align*}
we have a finite number of initial positions and velocities of the particles $(x_{0i},v_{0i})$ for $i=1,...,N$. It implies that there exists $R_1>0$ such that for all $R<R_1$, we have
\begin{align}\label{oddziel}
f_0|_{B_i(R)} = m_i\delta_{x_{0i}}\otimes\delta_{v_{0i}}
\end{align}
where $B_i(R)$ is a ball centred at $(x_{0i}, v_{0i})$ with radius $R$. In order to finish the proof it suffices to show that there exists $T^*$ such that 
\begin{align}\label{fd}
f^D(t) :=f(t)|_{B_{\underline{i}}(\frac{R}{4}) } = m_{\underline{i}}\delta_{x_{\underline{i}}(t)}\otimes\delta_{v_{\underline{i}}(t)},\  \ \mbox{for}\ \  t\in[0,T^*].
\end{align}
In other words, restriction of $f$ to a small ball centred initially in an atom (denoted by $f^D$) is precisely the said atom (it does not disperse). Denoting
\begin{align*}
f^C(t) := f(t) - f^D(t)
\end{align*}
%
%
we observe that $f^D$ and $f^C$ satisfy the following equation on $[0,T^*]$:
\begin{align}\label{weakfd}
\partial_t f^D + v\cdot\nabla_xf^D + \div [(F(f^C)+F(f^D))f^D] = 0.
\end{align}
Let us introduce
\begin{subnumcases}{\label{xva}}
 \ddt x_a(t)=v_a(t)\label{xva1}\\
\ddt v_a(t)= \int_{\R^{2d}}\psi(|x_a(t)-y|)(w-v_a(t))f^C(y,w,t) \mathrm{d} y\,\mathrm{d}w,\label{xva2}
\end{subnumcases}
with the initial data $(x_a(0),v_a(0))=(x_{0\underline{i}},v_{0\underline{i}})$. A critical property of the measure-valued solutions, that can be derived from the construction, is the local propagation of the support. Thus, since by definition $f^D$ is just a single particle at $t=0$, then even if it dissolves, for a short time it will still be contained within a cone of the form $x_{0\underline{i}}+tB(v_{0\underline{i}},\epsilon)$ for a small $\epsilon>0$. In particular $f^D$ remains separated from $f^C$, which allows to control the singularity of $\psi$. This ensures that the right-hand side of $\eqref{xva2}$ is smooth and thus (\ref{xva}) has exactly one smooth solution in $[0,T^*]$. Our goal is to 
show that $f^D$ is supported on the curve $(x_a(t),v_a(t))$ and that in fact (\ref{fd}) holds with $(x_{\underline{i}},v_{\underline{i}})\equiv(x_a,v_a)$.

We test (\ref{weakfd}) with $(v-v_a(t))^2$ and  with $|x-x_a(t)|^2t^{-1}$ getting, after some observations, that
\begin{align*}
\ddt\left(\int_{\r^{2d}}(t^{-1}f^D|x-x_a(t)|+f^D|v-v_a(t)|^2)\dx \dd v\right)+ \frac{1}{2}\int_{\r^{2d}}t^{-2}f^D|x-x_a(t)|^2\dx \dd v\\
\leq A(t)\int_{\r^{2d}}t^{-1}f^D|x-x_a(t)|^2+f^D|v-v_a(t)|^2\dx \dd v,
\end{align*}
with $A(t) \sim t^{-1/2-\alpha}$. In order to obtain the above inequality, we again control the singularity by ensuring the separation of $f^D$ and $f^C$ thanks to the cone-shaped propagation of $f^D$.

By Gronwall's lemma we get
\begin{align*}
\int_{\r^{2d}}(t^{-1}f^D|x-x_a(t)|^2+f^D|v-v_a(t)|^2)\dx\dd v\equiv 0
\end{align*}
on $[0,T^*]$. Thus on $[0,T^*]$ we have $x\equiv x_a$ and $v\equiv v_a$ on the support of $f$, 
which is exactly equivalent to (\ref{fd}) and the proof is finished.

We have proved that $f^D$ is mono-atomic. Then,  repeating the procedure for all atoms (the
number is finite) we conclude that $f$ is atomic on a time interval $[0,T^*]$ with possibly smaller,
but still positive $T^* > 0$. This procedure works till the first moment of sticking of an ensemble of
particles. To reach the moment of sticking of the particles we use regularity in time granted by Theorem \ref{main3}: $\partial_t  f \in L^p (0, T ; (C^1)^* )$ with $T > T_1$, where $T_1$ 
is the time of first sticking. Then we continue our procedure from $T_1$ till next time of sticking, up to the finial time of existence $T$.

\section{Singular Cucker-Smale model: hydrodynamics}\label{hydro}

\subsection{Formal derivation} 

Kinetic theory provides a way to model interacting particles on the {\it mesoscopic} scale level, it reduces the computational complexity. However, while the reduction is significant, we still end up with evolutionary equation in dimension equal to twice the dimension of the space, as  both $x$ and $v$ belong to $\R^d$. In case, when the number of particles is significantly large, the {\it macroscopic} scale can be employed instead. In such a scale the evolution of local 
in the phase space averages of density and velocity are studied. This way a hyrdodynamical limit is obtained and it further reduces the computational complexity of the model.

In case of the CS model, it can be understood by taking formally
\begin{align*}
f(t,x,v) = \rho(t,x)\otimes\delta_{u(t,x)}(v)
\end{align*}
in \eqref{cs-k} and testing it  with functions $\phi=1$ and $\phi=v$, respectively, obtaining the following continuity and momentum equations
\begin{subnumcases}{\label{macro}}
\partial_t\rho + \div(u\rho) = 0,\label{cont}\\
\partial_t (\rho u) + \div(\rho u\otimes u ) = \rho\int_{\Omega}(u(y)-u(\cdot))\psi(|y-\cdot|)\rho(y)\dd y.\label{mom}
\end{subnumcases}
Here $\Omega = \R^d$ or $\Omega = {\mathbb T}^d$. The system \eqref{macro} is a hydrodynamical version of the CS model, and we refer to it as the \emph{fractional Euler alignment system}.

We dedicate this section to the discussion of the recent contributions related to this system, focusing particularly on the singular case with weight \eqref{psing}. The past contributions, related mostly to the regular CS model, such as \cite{H-K-K1, H-K-K2, tan, carch} can be found in the survey article \cite{B-D-T1}. Here, we also mention a recent paper \cite{tad3} that follows the ideas from \cite{tan, carch} related to the existence of a critical threshold. 

Before we proceed, let us rewrite equation \eqref{mom} in a way that is often more convenient. Developing both terms on the right-hand side of \eqref{mom} and applying equation \eqref{cont} one obtains
\begin{align*}
\rho\partial_t u + \rho(u\cdot\nabla) u = \rho\int_{\Omega}(u(y)-u(\cdot))\psi(|y-\cdot|)\rho(y)\dd y,
\end{align*}
which then, divided by $\rho$, yields the velocity equation
\begin{align}\label{vel}
\partial_t u + (u\cdot\nabla) u = \int_{\Omega}(u(y)-u(\cdot))\psi(|y-\cdot|)\rho(y)\dd y.
\end{align}
Note that for
$$\rho=\tilde\rho\equiv const.$$
the right-hand side of \eqref{vel} is the fractional Laplacian of order $\frac{\gamma}{2}\in(0,1)$ defined as
\begin{align}\label{fraclap}
\Lambda^\gamma u(x)=(-\Delta)^\frac{\gamma}{2} u(x) := c\int_{\Omega}\frac{u(x)-u(y)}{|x-y|^{d+\gamma}}\dd y,
\qquad \gamma := \alpha-d\in (0,2),
\end{align}
which transforms \eqref{vel} into a variant of pressureless fractional Euler equation
\begin{align}\label{eul}
\partial_t u + (u\cdot\nabla) u + c\tilde\rho(-\Delta)^\frac{\gamma}{2}u = 0.
\end{align}
We emphasize the introduction of the exponent $\gamma = \alpha-d$ which helps to express the singularity in the hydrodynamic variant of the CS model. This of course puts $\alpha$ in the range $(d,d+2)$.
In equations similar to \eqref{eul} the regularity of solutions is granted by the fractional elliptic term. Therefore, to control the regularity of the fractional Euler alignment system \eqref{macro} when $\rho\neq const.$ the boundedness of the density $\rho$ from below is essential. It is a reflection of an ubiquitous, in hydrodynamics, problem of the control of the vacuum.

Rigorous derivation of the fractional Euler alignment system or other hydrodynamic limits of the Cucker-Smale kinetic equation (like in \cite{spoy}) is still mostly open. We refer to papers \cite{H-K-K1, H-K-K2} and to the survey \cite{B-D-T1}, noting that, for the most part, they deal with the model with the regular communication weight.

\subsection{Strong theory on ${\mathbb T}^1$}\label{Storus}

The most recent developments in the study of \eqref{macro} are due to the group Do {\it et.al.} \cite{kis} and independently due to Shvydkoy and Tadmor \cite{tad1, tad2, tad4}. The focus of their research is system \eqref{macro} in a 1D torus ${\mathbb T}$. The crucial discovery is that the quantity
\eq{\label{e}
e(t,x):= u_x(t,x)-\Lambda^\gamma\rho(t,x)=
 u_x(t,x) + \int_{\T}\psi(|x-y|)(\rho(t,x)-\rho(t,y))\dd y
}
for $\alpha=1+\gamma$,
satisfies the continuity equation
\begin{align}\label{magic}
e_t + (ue)_x = 0.
\end{align}
Thus, $q:=\frac{e}{\rho}$, defined for positive $\rho$ is transported with $u$, i.e.
\begin{align*}
q_t + uq_x = 0,
\end{align*}
and thus, in particular, its extrema are constant. Of course, such direct comparison between the regularity of $\rho$ and $u$ strongly relies on the one dimensional domain and as of today, its multidimensional generalization is unknown.

Discovery of the conservation law \eqref{magic} is the basis of multiple results for the fractional Euler alignment system in 1D torus. Methodology used in \cite{kis} and in \cite{tad1, tad2, tad4} varies strongly but the main difficulty boils down to the control of the minimal values of the density $\rho$ in \eqref{mom}. In \cite{kis, tad1}, quantity $e$ was used to provide a bound on the decay of $\rho$ of order $\frac{1}{1+t}$, which was sufficient for existence of smooth solutions and for flocking. However, as an example of application of quantity $e$ we shall present the uniform lower bound on the density obtained later in \cite{tad2}.

\begin{lem}[\cite{tad2}]\label{boundro}
Let $(u,\rho)$ be a smooth solution to \eqref{macro} on $\T$. Assume further that $\rho_0>0$ and that $|q_0|_\infty<\infty$. Then there exists a constant $C_0>0$ such that
\begin{align*}
\rho(t,x)\geq C_0\qquad\mbox{for all}\ x\in\T,\ t\geq 0.
\end{align*}
\end{lem}
\begin{proof}
Using the definition of $q$ we find that \eqref{cont} is equivalent to
\begin{align*}
\rho_t + u\rho_x = -q\rho^2 + \rho\int_\T\psi(|x-y|)(\rho(y)-\rho(\cdot))\dd y,
\end{align*}
which evaluated at $\rho_m(t):=\min_{x\in\T}\rho(t,x)$ leads to
\begin{align*}
\ddt\rho_m = -q\rho_m^2 + \underbrace{\rho_m\int_\T\psi(|x-y|)(\rho(y)-\rho_m)\dd y}_{\geq 0}\geq -q\rho_m^2 + \rho_m\int_\T\psi_m(\rho(y)-\rho_m)\dd y\\
\geq -(|q_0|_\infty+\psi_m)\rho_m^2 + {\mathcal M}\psi_m\rho_m,
\end{align*}
where $\psi_m:=\psi(1)=\inf_{x,y\in\T}\psi(|x-y|)$ and ${\mathcal M}:=\rho(\T)$. In the above inequality we replaced $q$ with $|q_0|_\infty$ by virtue of the fact that $q$ is transported so it retains the values of its extrema. Then we conclude that $\rho_m\geq \min\{\rho_m(0),\frac{{\mathcal M}\psi_m}{|q_0|_\infty+\psi_m}\}$.
\end{proof}

The control over the lower bound of $\rho$ provides the opportunity to prove further results. As explained at the beginning of this section the crucial application of the lower bound is in ensuring that the ellipticity granted by the right-hand side of \eqref{mom} or \eqref{vel} does not disappear. This ellipticity is used to obtain a variety of results, which we summarize in the following theorems.

\begin{theo}[\cite{kis}]
For $\gamma\in(0,1)$, the fractional Euler alignment system \eqref{macro} with periodic smooth initial data $(u_0,\rho_0)$, such that $\rho_0(x)>0$ for all $x\in \T$, has a unique global smooth solution.
\end{theo}

The above result is based on the lower bound on the density granted by the application of the quantity $e$ (however it should be noted that at the time of publication of \cite{kis}, Lemma \ref{boundro} was not known and the authors used a weaker, time dependent bound $\rho\gtrsim (1+t)^{-1}$). This allows to obtain local well-posedness in the form of a Beale-Kato-Majda type existence criterion, which reduces the global well-posedness to the question of regularity. Global regularity is then shown by a rescaling argument together with a modulus of continuity breaking-point method. The proof was first performed in the simplified case of $q\equiv 0$ and then generalized to $q\not\equiv 0$. Methods used in \cite{kis} are based on previous works due to Kiselev such as \cite{Kiselev}.

\begin{theo}[\cite{tad4}]
For $\gamma\in(0,2)$, the fractional Euler alignment system \eqref{macro} with periodic initial data $(u_0,\rho_0)\in H^3\times H^{1+\alpha}$, such that $\rho_0(x)>0$ for all $x\in \T$, has a unique global smooth solution $(u,\rho)\in L^\infty([0,\infty); H^3\times H^{\gamma})$. Moreover the solution converges exponentially fast to a flocking state
\begin{align*}
\bar{\rho}=\rho_\infty(x-t\bar{u})\in H^{\gamma}
\end{align*}
travelling with a finite speed $\bar{u}$ , so that for any $s<\gamma$ there exist $C=C_s$ and $\delta=\delta_s$ with
\begin{align*}
\|u(t)-\bar{u}\|_{H^3} + \|\rho(t)-\bar{\rho}(t)\|_{H^s}\leq Ce^{-\delta t},\qquad t>0,
\end{align*}
where $\bar{u}$ is the initial average velocity of the system.
\end{theo}

The proof of the above theorem can be found in \cite{tad4} but in reality, it is a final, rectified version of an effort started in \cite{tad1} and continued through \cite{tad2}. The proof of global well-posedness is based again on a Beale-Kato-Majda type existence criterion, this time however the issue of global regularity is addressed by controlling higher order derivatives of the solution and particularly -- their exponential decay. It results in a proof that is simpler than the proofs found in \cite{kis}, \cite{tad1} and \cite{tad2} (even though it relies on results due to Constantin and Vicol \cite{C-V} and Silvestre \cite{Silv}).

\subsection{Comparison with the PM equation -- theoretical results}\label{sec:PMtheory}
We will now summarize the results on a particular model closely related to system \eqref{macro}.
The starting point of this section is a reformulation of system \eqref{macro} used for example in \cite{kis}. Interestingly enough, the same reformulation was noticed and employed earlier in the context of compressible  fluids equations with density-dependent viscosity coefficients, see for example \cite{BD}. Taking formally $\gamma=2$ in \eqref{fraclap}, we get the following analogue of \eqref{magic} with $e=v_x$:
\eq{\label{vx}
(v_x)_t+(uv_x)_x=0,
}
where
\eq{\label{vur}
v_x=u_x-\Lambda^2\vr=u_x+\vr_{xx}.}
Therefore, for $v(t,x)$ vanishing at $x=\pm\infty$, equation \eqref{vx} leads to a simple transport equation for $v$, and so, the system \eqref{macro} takes the form
 \begin{equation}
\begin{cases}
\begin{aligned}
&\vr_t+(\vr u)_x=0\\
&v_t+uv_x=0.
\end{aligned}
\end{cases}
\label{kiselev}
\end{equation}
 A formal calculation shows that this system is equivalent to the pressureless compressible Navier-Stokes system with density dependent viscosity:
\begin{equation}
\begin{cases}
\begin{aligned}\label{main1}
&\vr_t+(\vr u)_x=0\\
&\lr{\vr u}_t+(\vr u^2)_x-(\vr^2 u_x)_x
=0.
\end{aligned}
\end{cases}
\end{equation}
As we shall see in Theorem \ref{theo1} below, the solution to this system, at least for conveniently chosen initial data, converges in some sense to the solution of the porous medium equation. The correspondence between both systems can be explained using the definition of $v$ \eqref{vur} again. It allows to rewrite system \eqref{kiselev} as
\begin{equation}
\begin{cases}
\begin{aligned}
&\vr_t-\lr{\frac{\vr^2}{2}}_{xx}=-(\vr v)_x\\
& (\vr v)_t+(\vr u v)_x=0.
\end{aligned}
\end{cases}
\label{main_v}
\end{equation}
In particular, the continuity equation  becomes the porous-medium (PM) equation with force, whose potential solves the continuity equation.

This observation was a core  of paper \cite{HaZa}, in which  the authors studied a generalization of  system \eqref{main1} on $\R$ augmented by the pressure term $\ep\lr{\vr^\gamma}_x$ with $\gamma>1$, and the initial condition
\eq{\label{initiald}
(\vr , \vr u)(0,x)=(\vr_{0},m_0)(x).}
They showed that if $(\sqrt{\vr} v)(0,x)=\frac{m_0}{\sqrt{\vr_0}}+\sqrt{\vr_0}\vr_{0x}=0$, then for $\ep\to 0$ the density solving the Navier-Stokes system converges to the solution of the porous medium equation 
\begin{equation}
\begin{cases}
\begin{aligned}
& \tilde{\vr}_t-\lr{\frac{\tilde\vr^2}{2}}_{xx}=0,\\
&\tilde\vr(0,x)=\vr_0(x)
\end{aligned}
\end{cases}
\label{porous}
\end{equation}
emanating from the same initial data. The same technique can be used to prove the existence of solutions to system \eqref{main1}, and to show that when $(\sqrt{\vr} v)(0,x)\to 0$, the same porous medium equation \eqref{porous} is recovered.
\begin{theo}
\label{theo1}
Let the initial data $(\vr_0,m_0)$ satisfy
\begin{equation}
\begin{aligned}
&\vr_0\geq 0, \, \vr_0\in L^1(\R)\cap L^\infty(\R), \quad
(\vr_0)_x^{\frac{3}{2}}\in L^2(\R),\\
&\frac{m_0^2}{\vr_0}\in L^1(\R),\quad \frac{|m_0|^{2+\kappa}}{\vr_0^{1+\kappa}}\in L^1(\R),\ \text{for\ some}\ \kappa>0,
\end{aligned}
\label{2.1}
\end{equation}
 and let
\eq{\left\|\frac{m_0}{\sqrt{\vr_0}}+\sqrt{\vr_0}\vr_{0x}\right\|_{L^2(\R)}\leq \eta.\label{zero_mass}}
1. System (\ref{main1}) with initial data \eqref{initiald} admits a global in time weak solution $(\vr_\eta, \sqrt{\vr_\eta}u_\eta)$, that is:
\begin{itemize}
\item The density $\vr_\eta\geq 0$ a.e., and the following regularity properties hold
$$
\begin{aligned}
&\vr_\eta\in L^\infty(0,T;L^1(\R))\cap {\mathcal C}([0,+\infty), (W^{1,\infty}(\R))^*),\\
&\lr{\vr_\eta^{\frac{3}{2}}}_{x}\in L^\infty(0,T;L^2(\R)),\ \sqrt{\vr_\eta}u_\eta \in L^\infty(0,T;L^2(\R)),
\end{aligned}
$$
where $(W^{1,\infty}(\R))^*$ is the dual space of $W^{1,\infty}(\R)$.
\item For any $t_2\geq t_1\geq 0$ and any $\psi\in C^1( [t_1,t_2]\times \R)$, the continuity equation is satisfied in the following sense:
\begin{equation}
\int_{\R}\vr_\eta\psi(t_2)\,\dx-\int_{\R}\vr_\eta\psi(t_1)\,\dx=
\int^{t_2}_{t_1}\int_{\R}(\vr_\eta\psi_t+\vr_\eta u_\eta\psi_x)\,\dxdt.
\label{2.3}
\end{equation}
Moreover, for  $\vr _\eta v_\eta=\vr_\eta u_\eta{+}\vr_\eta\vr_{\eta x}$, the following equality is satisfied
\begin{equation}
\int_{\R}\vr_\eta\psi(t_2)\,\dx-\int_{\R}\vr_\eta\psi(t_1)\,\dx=
\int^{t_2}_{t_1}\!\!\int_{\R}\lr{\vr_\eta\psi_t+\vr_\eta v_\eta\psi_x-\vr_\eta\vr_{\eta x}\psi_x}\,\dxdt.
\label{new2.3}
\end{equation}
\item For any $\psi\in C^{\infty}_{c}([0,T)\times\R)$ the momentum equation is satisfied in the following sense:
\begin{align}
\int_{\R}m_0\psi(0)\, \dx+\int^T_0\int_{\R}\lr{\sqrt{\vr_\eta}(\sqrt{\vr_\eta}u_\eta)\psi_t+(\sqrt{\vr_\eta}u_\eta)^2\psi_x}\,\dxdt\nonumber\\
-\langle\vr_\eta^2 u_{\eta x},\psi_x\rangle=0,
\label{2.4}
\end{align}
where the diffusion term is defined as follows:
\begin{equation}
\langle\vr_\eta^2 u_{\eta x},\psi_x\rangle=-\int^T_0\!\!\int_{\R}\vr_\eta^{\frac{3}{2}}\sqrt{\vr_\eta}u_\eta\psi_{xx}\, \dxdt-\frac{4}{3}\int^T_0\!\!\int_{\R}\lr{\vr_\eta^{\frac{3}{2}}}_x\sqrt{\vr_\eta} u_\eta\psi_x\, \dxdt.
\label{2.5}
\end{equation}
\end{itemize}

2. For $\eta\to 0$, $\vr_\eta$ converges strongly to $\tilde\vr$ -- the strong solution to the porous medium equation \eqref{porous}
in the following sense: there exists a constant $C>0$ depending on $\vr_0$ such that
\begin{equation}
\|(\tilde\vr-\vr_\eta)(t)\|_{H^{-1}(\R)}\leq C \eta^{\frac{1}{2}}t^{\frac{1}{2}}.
\label{H-1}
\end{equation}
\end{theo}

Note that in  assumptions of this theorem we only require that $\vr_0\geq 0$, in particular, the initial density can be compactly supported. Then, the natural question is how does the support of the density propagate in time. For the degenerate Navier-Stokes equations the answer to this question is only partial, see for example \cite{YZ, JXZ}. Estimate \eqref{H-1} allows us to compare  weak solutions to the Navier-Stokes system with  strong solutions of the porous medium equation. From the classical theory we know that the interface between fluid and the vacuum for the latter moves with  the finite speed. For the special class of self-similar solutions, the so-called Barenblatt solutions, one can even give exact formula for the velocity of this motion, see \cite{Va}. 

{\it Idea of the proof.}
The proof of existence of weak solutions starts from the approximate  Navier-Stokes system augmented by the artificial viscosity term:
\begin{equation}
\begin{cases}
\begin{aligned}
& \vr_{\ep t}+(\vr_\ep u_\ep)_x=0\\
&\lr{\vr_\ep u_\ep}_t+(\vr_\ep u_\ep^2)_x- (\vr_\ep^2+\ep\vr_\ep^\theta u_{\ep x})_x
=0\\
&\vr_\ep(0,x)=\vr_{\ep,0}(x),\ \vr_\ep(0,x)u_\ep(0,x)=m_{\ep,0}(x),
\end{aligned}
\end{cases}
\label{mainpp}
\end{equation}
where $\ep>0$, $\theta\in (0,1/2)$ are constant. Further, we cut the domain into a bounded interval $\Omega=[-M,M]$ for $M$ large and we supplement system \eqref{mainpp} with the boundary conditions
$u_\ep|_{\partial\Omega}=0.$
We assume that $\vr_{\ep,0}>C(\ep)>0$, and that $\vr_{\ep,0}$, $m_{\ep,0}$ converge to $\vr_0$, $m_0$ in the following sense
\begin{equation}
\begin{aligned}
&\vr_{\ep,0}\to\vr_0, \quad\text{strongly in } L^1(\Omega), \\
&\lr{\vr_{\ep,0}^{\frac{3}{2}}}_x\to\lr{ \vr_0^{\frac{3}{2}}}_x\quad\text{strongly in } L^2(\Omega)\\
&\frac{m_{\ep,0}^2}{\vr_{\ep,0}}\to\frac{m_0^2}{\vr_0}\quad\text{strongly in } L^1(\Omega),\\
&\frac{|m_{\ep,0}|^{2+\kappa}}{\vr_{\ep,0}^{1+\kappa}}\to\frac{|m_0|^{2+\kappa}}{\vr_0^{1+\kappa}}\quad\text{strongly in } L^1(\Omega).
\end{aligned}
\label{new2.1}
\end{equation} The existence of the classical solutions for $\ep, M$ being fixed can be obtained as in
the works of   Q. Jiu, Z. Xin  \cite{Jiu}, and of  H.-L. Li, J.~Li, and Z.~Xin  \cite{LLX} that treat the full Navier-Stokes system including the pressure term.
The most important here are the lower and upper estimates of the density that however do not depend on the presence of the pressure. The way to obtain them is to rewrite system \eqref{mainpp} in the mass Lagrangian coordinates 
\begin{equation}\label{Lag_ch}
y= \int_{-M}^x{\vr(\tau,s)ds}, \qquad \tau= t.
\end{equation}
Since the total mass is given and bounded we have 
\begin{equation}\label{init2}
\int_{-M}^M\vr(t,s){\rm{d}}s=L<\infty,
\end{equation}
 and so $y\in[0,L]=\Omega_L$.  Using \eqref{Lag_ch}, system \eqref{main1} may be transformed into the following one
\begin{equation}\label{main_Lag}
\left\{
\begin{array}{l}
\vr_\tau + \vr^2 u_y=0\\
u_\tau- ((\vr ^3+\ep\vr^{1+\theta})u_y)_y  =0\\
\vr_\ep(0,y)=\vr_{\ep,0}(y),\ \vr_\ep(0,y)u_\ep(0,y)=m_{\ep,0}(y),
\end{array}
\right.
\end{equation}
where the boundary conditions are now equal to
$u|_{\partial\Omega_L}=0.$
Then, for regular solutions $\vr_\ep, u_\ep$ of \eqref{main_Lag} s.t. $\vr_\ep>0$, we have the following a-priori estimates:
 \begin{itemize}
 \item[i)] the energy estimate in the Lagrangian coordinates
 \begin{equation}
\begin{aligned}
\intOL{ \frac{u_\ep^2}{2}(T)}+\intTOL{\vr_\ep(\vr_\ep^2+\ep\vr_\ep^\theta)(u_{\ep\eta})^2}\leq \intOL{\frac{u_\ep^2}{2}(0)};
\end{aligned}
\label{energyL}
\end{equation}
\item[ii)] the entropy estimate in the Lagrangian coordinates
 \begin{equation}
\begin{aligned}
&\intOL{\lr{u_\ep^2+(\vr^2_\ep)^2_\eta+\ep^2(\vr_\ep^\theta)^2_\eta}(T)}+\intTOL{\vr_\ep(\vr_\ep^2+\ep\vr_\ep^\theta)(u_{\ep\eta})^2}\\
&\hspace{4cm}\leq \intOL{\lr{u_\ep^2+(\vr^2_\ep)^2_\eta+\ep^2(\vr_\ep^\theta)^2_\eta}(0)};
\end{aligned}
\label{entropyL}
\end{equation}
 \end{itemize}
As a consequence of these estimates we can show the following lemma.
\begin{lem}
Let $\vr_\ep, u_\ep$ be a regular solution of \eqref{main_Lag}, satisfying \eqref{energyL} and \eqref{entropyL}. Then there exists constants $C$ and $C(\ep)$ such that
\eq{
0<C(\ep)\leq\vr_\ep\leq C.
\label{ud}
}
\end{lem}
For the proof of this Lemma we refer, for example, to \cite{Jiu}, Lemma 3.3. Having these estimates at hand,  proving the existence of regular solutions to \eqref{main_Lag} is a classical result,
see for example {\cite{Alessandra}} Chap 7.

The existence of smooth solutions to \eqref{main_Lag} allows to come back to the system written in the Eulerian coordinates \eqref{mainpp} for $\ep,M$ being fixed, and to translate the estimates from above to the following ones
 \begin{itemize}
 \item[i')~] the energy estimate in the Eulerian coordinates
 \begin{equation}
\begin{aligned}
&\intOM{\frac{\vr_\ep u_\ep^2}{2}(T)}+\intTOM{(\vr_\ep^2+\ep\vr_\ep^\theta) u_{\ep x}^2}\leq \intOM{\lr{\vr_\ep \frac{u_\ep^2}{2}+\frac{\ep}{\gamma-1}\vr_\ep^{\gamma}}(0)};
\end{aligned}
\label{energieuni2}
\end{equation}
\item[ii')~] \ the entropy estimate in the Eulerian coordinates
 \begin{equation}
\begin{aligned}
&\intOM{\vr_\ep \frac{\lr{u_\ep+(1+\ep\vr_\ep^{\theta-2})\vr_{\ep x}}^2}{2}(T)}\leq \intOM{\vr_\ep \frac{\lr{u_\ep+(1+\ep\vr_\ep^{\theta-2})\vr_{\ep x}}^2}{2}(0)}.
\end{aligned}
\label{energieuni}
\end{equation}
 \end{itemize}
 In addition to that, as in the work of  A. Mellet and A. Vasseur  in \cite{MV07}, 
 one can improve the uniform estimates of the velocity vector field
\eq{\label{testMV}
\intOM{\frac{\vr_\ep|u_\ep|^{2+\kappa}}{2+\kappa}(T)}+(\kappa+1)\intTOM{(\vr_\ep^2+\ep\vr_\ep^\theta)|u_\ep|^\kappa| u_{\ep x}|^2}
\leq\nonumber\\
\leq \intOM{\frac{\vr_\ep|u_\ep|^{2+\kappa}}{2+\kappa}(0)},}
for some $\kappa>0$.
These estimates lead to the following bounds
\eq{
&\|\sqrt{\vr_\ep}u_\ep\|_{L^\infty(0,T; L^2(\Omega))}+\|(\vr_\ep^{\frac{3}{2}})_x\|_{L^\infty(0,T;L^2(\Omega))}
+\ep\| (\vr_\ep^{\theta-\frac{1}{2}})_x\|_{L^\infty(0,T;L^2(\Omega))}\\
&+\|\vr_\ep|u_\ep|^{2+\kappa}\|_{L^\infty(0,T; L^1(\Omega))}+
+\|(\vr_\ep+\sqrt{\ep}\vr_\ep^{\theta/2})u_{\ep x}\|_{L^2(0,T; L^2(\Omega))}\leq C.
\label{BDrho}}
Moreover, translating \eqref{ud} into the Eulerian coordinates, uniformly in $\ep$ we have
\begin{equation}
\|\vr_\ep\|_{L^\infty(0,T; L^\infty(\Omega))}\leq C.
\label{technorm}
\end{equation}
%
With these estimates compactness arguments from \cite{MV07}, \cite{LLX} yield convergence as $\ep\to 0$ of the approximate solution $(\vr_\ep, u_\ep)$ to the weak solution $(\vr, u)$ specified in Theorem \ref{theo1}, on the domain $\Omega=[-M,M]$ with the no-slip boundary condition for $u$. In order to let $M\to \infty$, one can combine the diagonal procedure with the convergence of the initial data $(\vr_{\ep,0},m_{\ep,0})\to (\vr_0,m_0)$ as it was done in \cite{Jiu}.

\medskip
The proof of the second part of Theorem \ref{theo1} corresponds to the pressureless limit studied in \cite{HaZa}. We recall the main steps. Recalling the notation $v_\eta=u_\eta+\vr_{\eta x}$, the first part of Theorem \ref{theo1} provides that the continuity equation
\eq{
\vr_{\eta t}-\frac{1}{2}\lr{\vr^2_\eta}_{xx}+(\vr_\eta v_\eta)_x=0,
\label{5}
}
it is satisfied in the sense of distributions.
From \eqref{zero_mass} and \eqref{energieuni} with $\ep=0$ it also follows that
\begin{equation}
\sup_{t\in[0,T]}\|\sqrt{\vr_\eta}v_\eta(t)\|_{L^2(\R)}\leq\eta.
\label{enerate}
\end{equation}
Therefore, when $\eta\to0$, we expect to show that the last term in equation \eqref{5} converges in the "$H^{-1}$" sense to the strong solution of the corresponding porous-medium equation  \eqref{porous} with the same initial data $\vr_0$. 
Rigorous proof of this fact is based to the duality technique  in the spirit of J. L. V\'azquez (see \cite{Va} Section 6.2.1), that was used in \cite{HaZa} for the pressureless limit. 

The convergence  $\vr_\eta\to \tilde\vr$ as stated in \eqref{H-1} is relatively weak. Note that in \cite{HaZa} the convergence of weak solutions to Navier-Stokes  system with viscosity coefficient of the form $\vr^\alpha$ could have been significantly improved for $1<\alpha\leq\frac{3}{2}$. This was possible thanks to uniform boundedness of $\vr_{\eta x}$ in $L^\infty(0,T; L^2(\R))$.  In this case one can  estimate the mass corresponding to $\vr_\eta$ inside the support of  certain Barenblatt profile.  This provides a partial information about evolution of the interface between the medium and  vacuum. For further discussion on that matter we refer the reader to \cite{HaZa}, Section 5,  and to \cite{PAM} for relevant results in the multi-dimensional case.

\subsection{Comparison with the PM equation -- numerical illustration}

In this Section we present the results of numerical simulations for the pressureless Navier-Stokes system with density dependent viscosity \eqref{main1} and the porous medium equation \eqref{porous}.  We aim to illustrate analytical developments of Section \ref{sec:PMtheory}  and demonstrate the evolution of $\vr, u,$ and $\tilde\vr$ with respect to various initial conditions. For the sake of numerical simulations we will always assume that $\vr_0>0$, then the initial velocity $u_0$ may be extracted from $m_0$, moreover
$$v_0= v(0,x)=\frac{m_0}{\vr_0}+\vr_{0x}.$$  
We define parameter $c$ as a ratio
\begin{equation}\label{c}
 \frac{m_0}{\vr_0}=c\,\vr_{0x}.
\end{equation}
Note that for $c=-1$, $v_0=0$ and thus we expect that both approximate solutions to \eqref{main_v} and \eqref{porous} coincide.

Instead of the whole $\R$, the computational domain is approximated by sufficiently large interval $\Omega = [-20,20]$, and we employ zero Neumann boundary condition on $\partial\Omega$, namely $ u_x\cdot \boldmath{n} = 0$.
As an initial density profile we choose smooth functions with compact support;
two initial conditions are taken into account:
$$
\vr^1_0 = \frac{0.2}{1+x^2},\quad
\vr^2_0 = \frac{0.2}{1+(x-10)^2} + \frac{0.2}{1+(x+10)^2},$$
that are refereed to as Case 1 and Case 2, respectively.  As a consequence of \eqref{c} the initial velocity reads:
$$u^1_0 = -c\frac{0.4x}{(1+x^2)^2},\quad u^2_0 = -c \frac{0.4(x-10)}{(1+(x-10)^2)^2} - c\frac{0.4(x+10)}{(1+(x+10)^2)^2}. $$

For the spatial discretization, standard finite element method is employed where the discrete space for the density is one order higher then for the velocity, namely $(\rho_h,u_h)\in \mathbb{P}^3(\Omega)\times\mathbb{P}^2(\Omega)$, where $\mathbb{P}^k(\Omega)$ denotes continuous Lagrange element of order $k$. In time we use implicit time-stepping, with time step $\Delta t = \Delta x = 0.01$.  The nonlinear problem is solved by means of the Newton method with the Jacobian computed by automatic differentiation. Implementation is based on the Finite Element library FEniCS. 

To illustrate the dependence of the solutions to \eqref{main1} and \eqref{porous} on the initial value of $v_0$, the constant $c$ has been chosen from a set 
\eq{\label{cValue}
c \in \{-0.1,-0.5. -0.9, -1.0 ,-1.1,-1.5,-1.9 \}.}
This corresponds to 
$$v_0=0.9\vr_{0x}, 0.5\vr_{0x}, 0.1\vr_{0x}, 0, -0.1\vr_{0x},-0.5\vr_{0x},-0.9v_0.$$

Figures \ref{fig:case1} and \ref{fig:case2} demonstrate time development of the initial profiles, $\vr^1_0$ and $\vr^2_0$, respectively. For each of the values of parameter $c$ we depict the profile of the density $\vr$, the velocity $u$, and the solution to the porous medium equation $\tilde\vr$ at times $t=0,\ 200,\ 400$. Note that at time $t=0$, the graphs of $\vr$ and $\tilde\vr$ coincide.
In addition to that, in Figures \ref{fig:rho1} and \ref{fig:rho2} we compare density profiles at time $t=200$ for various values of the parameter $c$. As predicted by the theoretical considerations from the previous section, for $c\to-1$, which corresponds to $v_0\to0$, the profile of the density $\vr$ resembles more and more the profile of $\tilde\vr$, at least up to some time \eqref{H-1}. For $c=-1$, i.e. for $v_0=0$, the profiles of $\vr$ and $\tilde\vr$ coincide. What is also interesting, but not captured by the current analytical theory, is the behaviour of the solution for $c\not\to -1$. On one hand, we have
 creation of high concentration around the origin for $c>-1$ (see the top rows of Figures \ref{fig:case1}, \ref{fig:case2}, and the left parts of Figures \ref{fig:rho1} and \ref{fig:rho2}). On the other, there is a separation of the initial mass into two groups travelling in the opposite directions for $c<-1$ (see the bottom rows of Figures \ref{fig:case1}, \ref{fig:case2}, and the right-hand parts of Figures \ref{fig:rho1} and \ref{fig:rho2}). In the first scenario ($c>-1$), it seems that depending on the initial data, the final state may consist of more than one concentration picks (see the top row of Figure \ref{fig:case2} and the left part of Figure \ref{fig:rho2}). While for the second scenario ($c<-1$), after initial division, the parts of mass colliding around the origin accumulate and create steep profile (see bottom row of Figure \ref{fig:case2} and the right part of Figure \ref{fig:rho2}).

\begin{figure}[h!]
	\begin{center}
		\includegraphics[width=1.\textwidth]{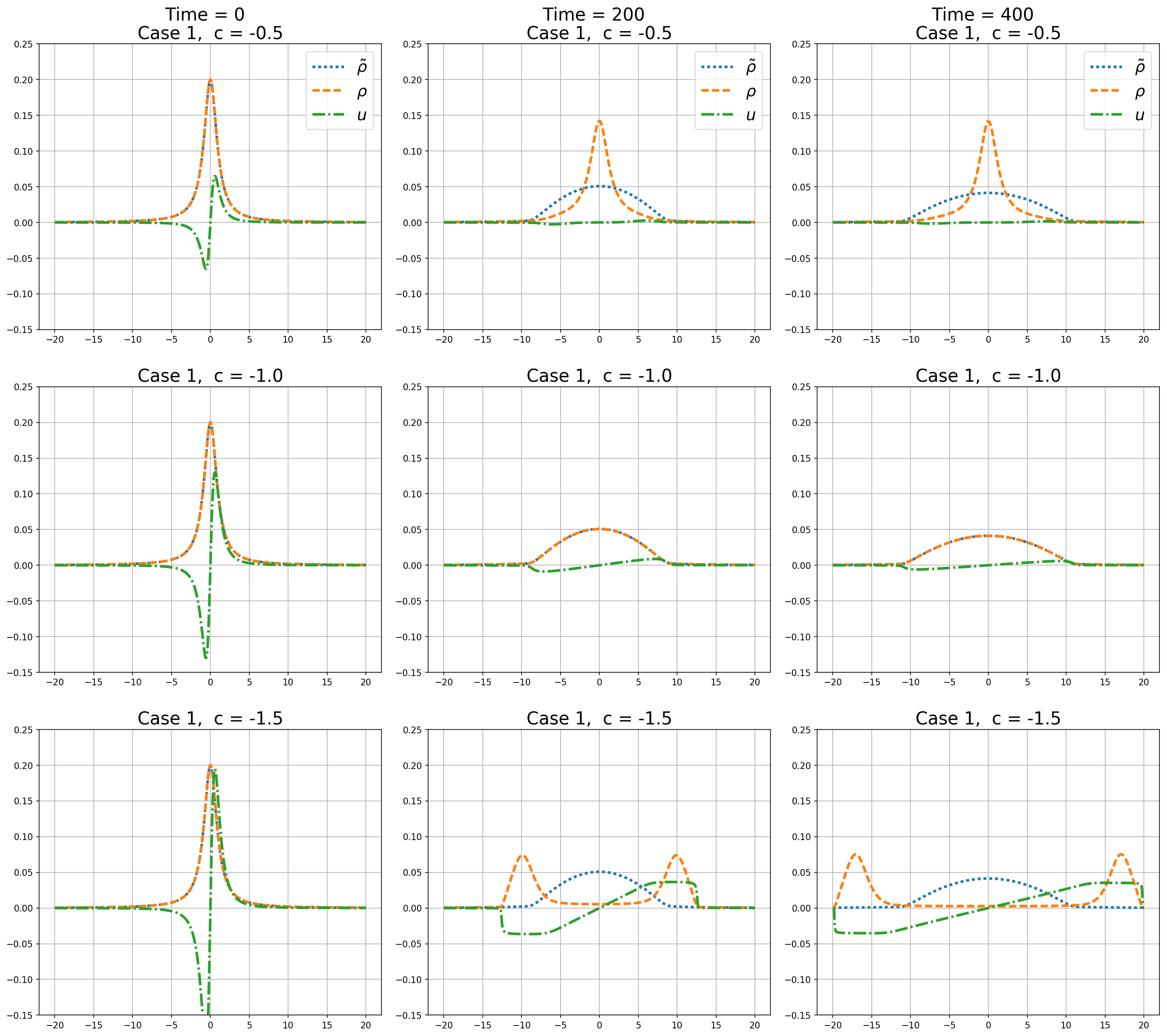}
	\end{center}	
	\caption{Case:1 Evolution of $\vr$, $u$ solving \eqref{main1} and of $\tilde\vr$ solving \eqref{porous}, for $\vr_0=\vr_0^1$ and $u_0=u_0^1$. Comparison of the profiles for $t=0$, $t=200$, and $t=400$, for different values of parameter $c$ as specified by \eqref{cValue}.}
	\label{fig:case1}
\end{figure}

\begin{figure}[h!]
	\begin{center}
		\includegraphics[width=1.\textwidth]{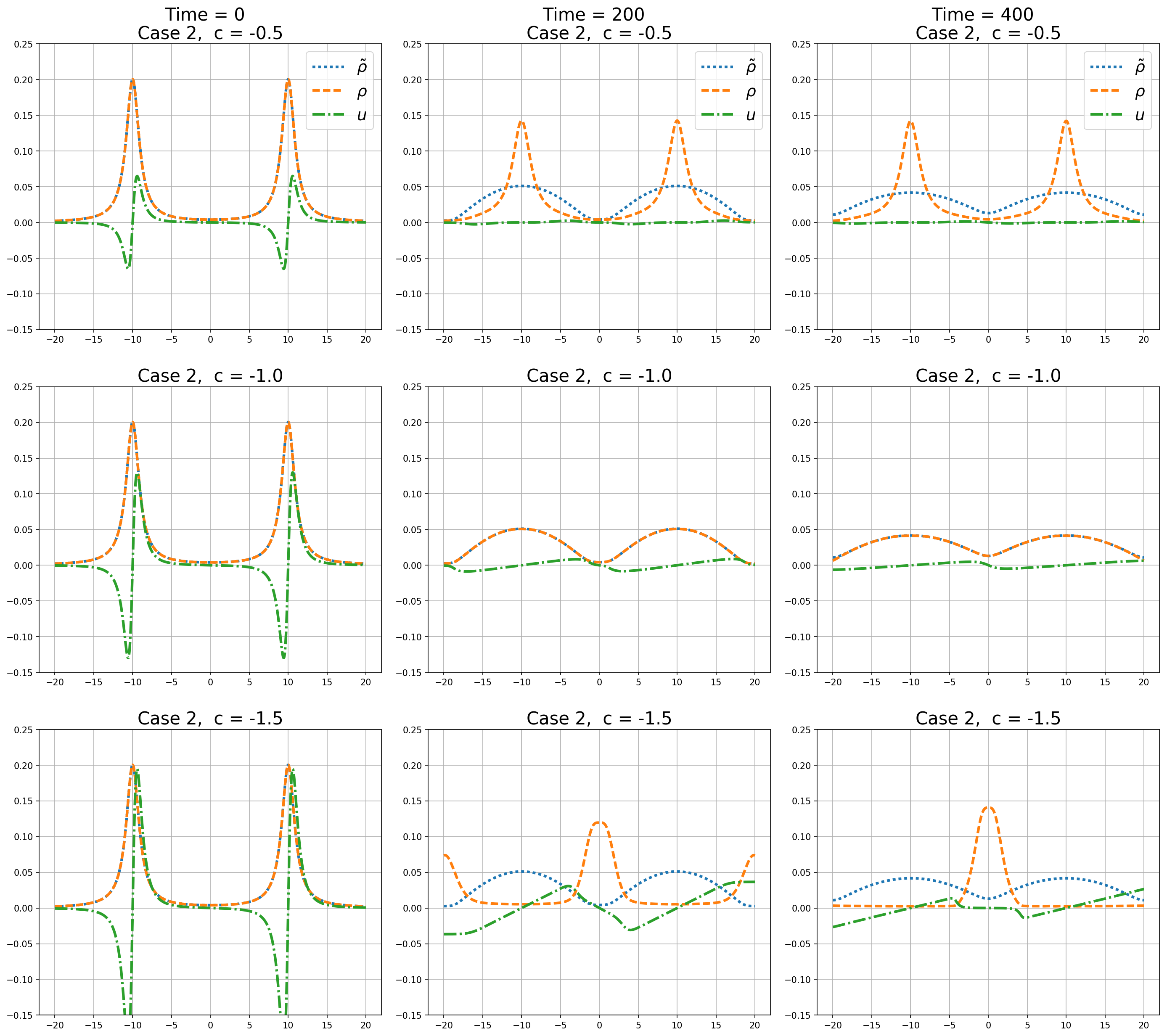}
	\end{center}	
\caption{Case 2: Evolution of $\vr$, $u$ solving \eqref{main1} and of $\tilde\vr$ solving \eqref{porous}, for $\vr_0=\vr_0^2$ and $u_0=u_0^2$. Comparison of the profiles for $t=0$, $t=200$, and $t=400$, for different values of parameter $c$ as specified by \eqref{cValue}.}
	\label{fig:case2}
\end{figure}

\begin{figure}[h!]
	\begin{center}
		\includegraphics[width=.75\textwidth]{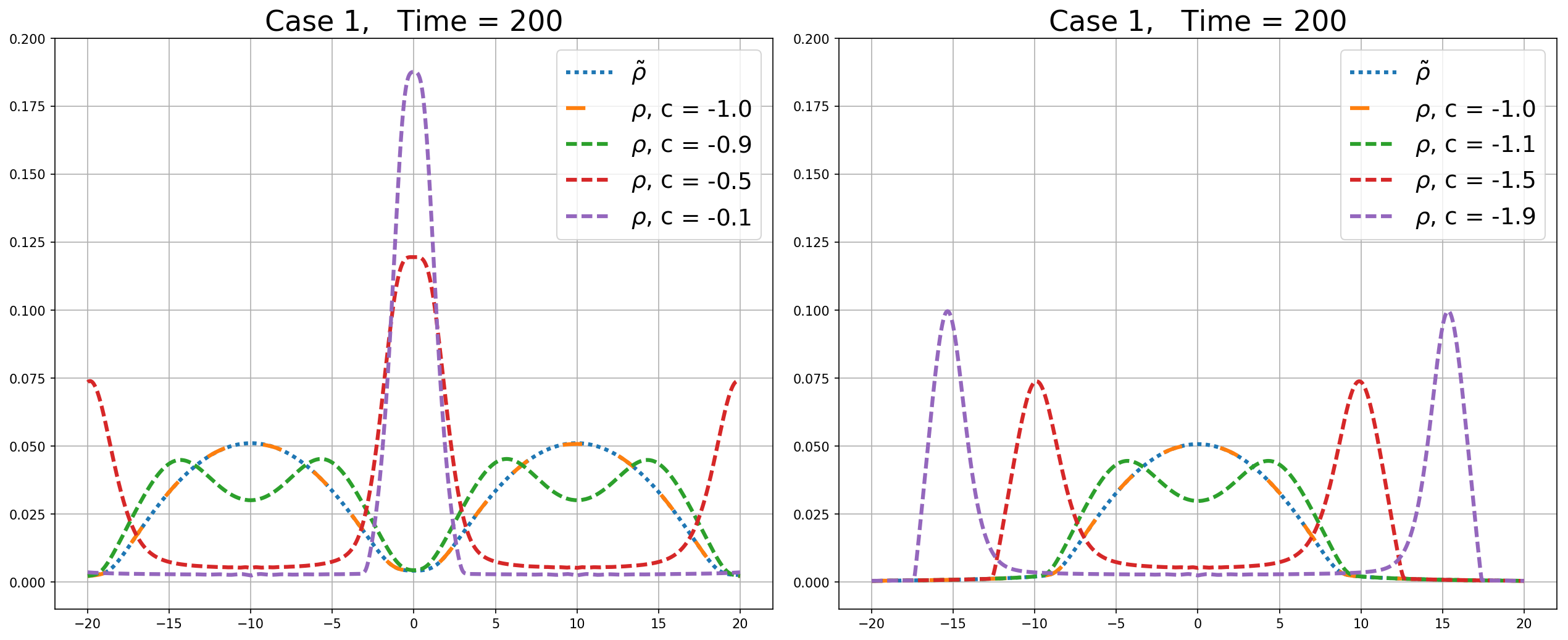}
	\end{center}	
	\caption{Comparison of the density profiles for Case 1: $\vr_0=\vr_0^1$, $u_0=u_0^1$, at $t=200$.}
	\label{fig:rho1}
\end{figure}

\begin{figure}[h!]
	\begin{center}
		\includegraphics[width=.75\textwidth]{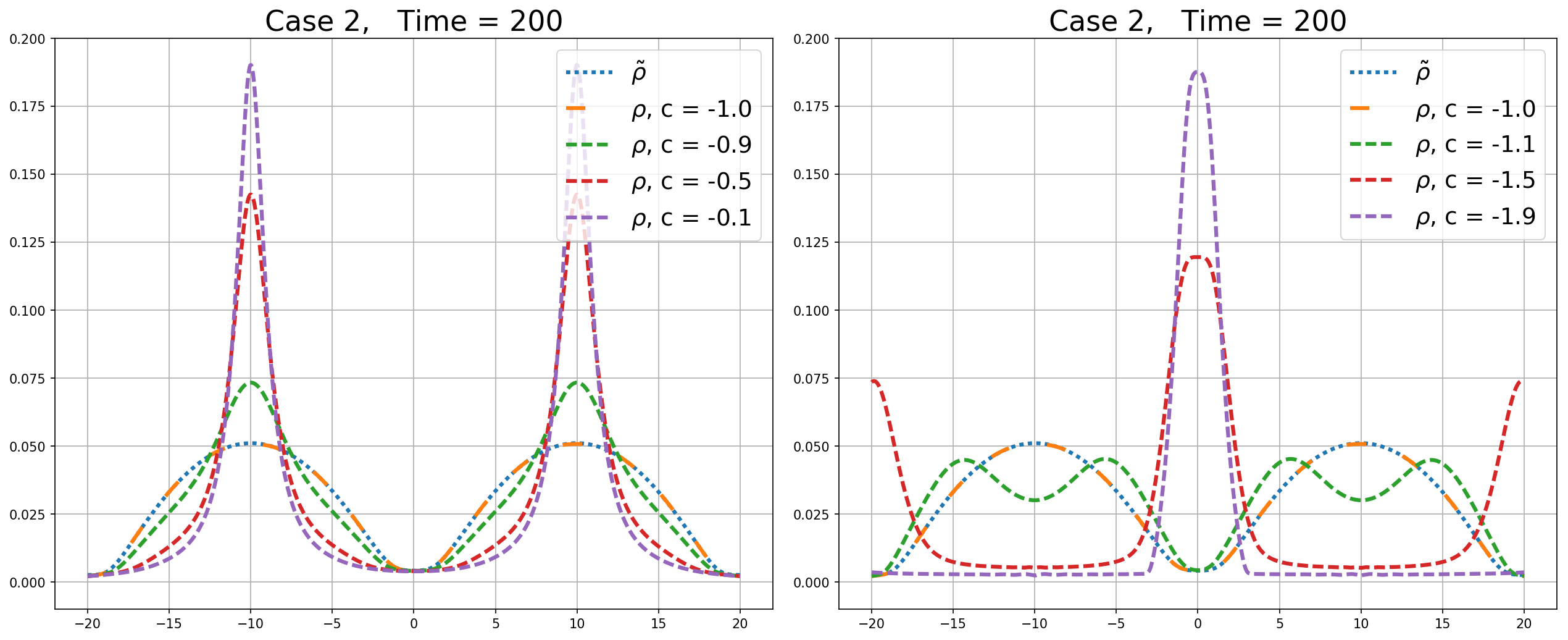}
	\end{center}	
	\caption{Comparison of the density profiles for Case 2: $\vr_0=\vr_0^2$, $u_0=u_0^2$, at $t=200$.}
	\label{fig:rho2}
\end{figure}

\subsection{Strong theory in $\T^d$ with small data}\label{shvyd}
Generalization of the results by Do {\it et. al.} and Shvydkoy and Tadmor, presented in Section \ref{Storus}, to higher dimensions remains open. It is unclear what object should replace quantity $e$ from \eqref{magic} and the intuitively natural candidate
\begin{align*}
e = \div_x u -\Lambda^\gamma\rho(t,x)
\end{align*}
does not satisfy the continuity equation like in the case of $d=1$. However it still satisfies equation
\begin{align}\label{ed}
e_t + \div_x (ue) = (\div_x u)^2 - {\rm Tr}(\nabla u)^2,
\end{align}
which can be used to derive an estimate on $e$. Such approach was applied by Shvydkoy in \cite{Shv} leading to the following result.

\begin{theo}[\cite{Shv}]
Let $\gamma\in(0,2)$. There exists an $N\in{\mathbb N}$ such that for any sufficiently large $R>0$ any initial condition $(u_0,\rho_0)\in H^m(\T^d)\times H^{m-1+\gamma}(\T^d)$, $m\geq d +4$, satisfying
\begin{align}\label{shvas}
|\rho_0|_\infty, |\rho_0^{-1}|_\infty, [u_0]_{\dot{W}^{3,\infty}}, [\rho_0]_{\dot{W}^{3,\infty}}\leq R,\qquad \sup_{x,y\in\T^d}|u_0(x)-u_0(y)|\leq \frac{1}{R^N},
\end{align}
gives rise to a unique global solution in class $C([0,\infty):H^m(\T^d)\times H^{m-1+\gamma})$. Moreover, the solution converges to a flocking state exponentially fast. Here $\dot{W}^{3,\infty}$ is the homogeneous Sobolev space of functions with third weak derivative belonging to $L^\infty$ and $[\cdot]_{\dot{W}^{3,\infty}}$ is its seminorm.  
\end{theo}
Conditions \eqref{shvas} should be viewed as assumptions on smallness of initial data in terms of initial deviation of the velocity $u$ from the average. On top of that it is assumed that $\rho_0$ does not have singularities and is separated from $0$.

{\it Idea of the proof.}
The proof is performed in the spirit of \cite{tad4}. Local existence, yet again, is reduced to the control over $|\nabla u|_\infty$ by a Beale-Kato-Majda type criterion. To prolong the existence a global estimate is required and to get it, similarly to the $d=1$ case, the author makes sure that $\rho$ is separated from $0$. It is achieved using quantity $e$ and thus the proof is based on its estimation. Such estimation is performed using equation \eqref{ed}, which enables to bound $e$ in terms of $|\nabla u|_\infty$ and itself. To close the estimate smallness condition \eqref{shvas} is needed. Once existence in a sufficiently high class of regularity is established, flocking follows using elementary arguments based on the structure of the CS model (the argumentation is in essence the same as in the simplest case of the CS particle system with regular communication weight).\hfill$\square$

\subsection{Strong theory in $\R^d$ with small data}
A mostly different approach is employed in \cite{dmpw} to deal with the case of $\Omega=\R^d$ with $d\geq 2$. The idea is to use the classical theory of fractional elliptic hydrodynamics. The starting point is the observation that the right-hand side of equation \eqref{vel} is, up to a constant, the fractional Laplace operator of $u$. To pursue this idea we assume, similarly to the $\T^d$ case in Section \ref{shvyd}, that $\rho_0$ and $u_0$ only slightly deviate from constants. Then smallness of the velocity is preserved in time and ensures that the density $\rho$ also stays close to a constant for all times. To give a proper mathematical description of this approach let us rewrite the velocity equation \eqref{vel} obtaining
\begin{align}\label{burg}
\partial_t u + (u\cdot\nabla)u + c_\alpha (-\Delta)^\frac{\gamma}{2}u = {\mathcal B},
\end{align}
where the reminder is defined as
\eqh{
& {\mathcal B}:= (1-\rho)(-\Delta)^\frac{\gamma}{2}u + {\mathcal I},\\
& {\mathcal I}:=\int_{\R^d}\frac{u(x+h)-u(x)}{|x|^{d+\gamma}}(\rho(x+h)-\rho(x))dh.
}
Then the system \eqref{cont} + \eqref{burg} can be viewed as a compressible fractional Burgers system with a (hopefully manageable) right-hand side. Here also the effect of taking $\rho$ close to a constant is apparent since then both terms that constitute ${\mathcal B}$ become small. Based on the works on fractional Burgers equation, such as \cite{M-W} by Miao and Wu, it is reasonable to expect that the range of admissible $\gamma$ would be $(0,2)$. However for the sake of simplicity and accessibility we restrict the range of admissible singularity $\gamma$ to $(1,2)$ which allows us to simplify the system even further by transferring the convection to the right-hand side:
\begin{align}\label{heat}
\partial_t u  + c_\alpha (-\Delta)^\frac{\gamma}{2}u = {\mathcal R},
\end{align}
with
\begin{align*}
{\mathcal R}:= (1-\rho)(-\Delta)^\frac{\gamma}{2}u + {\mathcal I} -(u\cdot\nabla)u.
\end{align*}
Such presentation of the velocity equation provides an opportunity to  consider the fractional Euler alignment system as a compressible fractional heat equation, whose theory boast a wide range of tools to chose from. We chose to use the Besov framework. Thus the fractional Euler alignment system is reduced to a well understood problem of compressible heat equation in the language of Besov spaces, with most of the difficulty moved to the external force ${\mathcal R}$ on the right-hand side. Such approach produced the following theorem.
\begin{theo}\label{main-intro}\cite{dmpw}
Assume that  $\gamma \in (1,2)$ and consider initial data $(\rho_0,u_0)$ so that $u_0$ and $\nabla u_0$ are  in 
$\dot B^{2-\gamma}_{d,1},$ and $\rho_0 -1$ and $\nabla\rho_0$ are  in $\dot B^{1}_{d,1}.$ 
 There exists $\varepsilon>0$ such that  if  in addition 
\begin{equation}\label{eq:smallness}
\|u_0\|_{\dot B^{2-\gamma}_{d,1}} +  \|\rho_0-1 \|_{\dot B^{1}_{d,1}}  <\varepsilon,
\end{equation}
then the fractional Euler system $\eqref{cont} + \eqref{mom}$
has a unique global solution $(\rho,u)$ such that
$$u,\nabla u \in {\mathcal C}_b(\R_+; \dot{B}^{2-\gamma}_{d,1}) \cap L^1(\R_+;\dot B^{2}_{d,1})
\andf(\rho -1),\nabla\rho \in {\mathcal C}_b(\R_+; \dot B^{1}_{d,1}).$$
In the case where the smallness condition is fulfilled only by $\rho_0$, there exist 
a unique solution $(\rho,u)$ on some time interval $[0,T]$ with $T>0$ so that
$$u,\nabla u \in {\mathcal C}_b([0,T]; \dot{B}^{2-\gamma}_{d,1}) \cap L^1([0,T];\dot B^{2}_{d,1})
\andf(\rho -1),\nabla\rho \in {\mathcal C}_b([0,T]; \dot B^{1}_{d,1}).$$
\end{theo}
The proof of the above theorem follows by a standard iterative scheme with application of Besov techniques in fractional heat equation. The fractional laplacian in ${\mathcal R}$ is dealt with by the fractional laplacian on the left-hand side thanks to the smallness of $\rho-1$. The main difficulty lies in the control of ${\mathcal I}$, which is shown to satisfy the inequality
\[ \| {\mathcal I} \|_{\dot B^{2- \gamma}_{d,1}} \leq C \|\nabla u\|_{ \dot B^{1}_{d,1}} \|\rho-1\|_{\dot B^{1}_{d,1}} \ . \]  

Finally information provided by Theorem \ref{main-intro} leads to the following corollary regarding the asymptotic behavior of the solutions.
\begin{coro}
 Let $(\rho,u)$ be a global in time solution given by Theorem \ref{main-intro}. Then 
 $$
 \|u(t)\|_{L^\infty} \to 0 \mbox{ \ \ as \ \ } t \to \infty.
 $$
\end{coro}

Here, let us note that, while the information provided by Theorem \ref{main-intro} enables the conclusion of the asymptotic decay of velocity, it is insufficient to ensure the exponential rate of the decay, since we are on the whole space $\R^d$ (unlike the ${\mathbb T}^d$ case presented in Section \ref{shvyd}).

\footnotesize
\section*{Acknowledgement}
PM was supported by the Deutsche Forschungsgemeinschaft (DFG, German Research Foundation) - 314838170, GRK 2297 MathCoRe. JP was supported by the Polish MNiSW grant Mobilno\' s\' c Plus no. 1617/MOB/V/2017/0. 
EZ was supported by the UCL Department of Mathematics Grant, grant Iuventus Plus  no. 0888/IP3/2016/74 of Ministry of Sciences and Higher Education RP, and by the Simons - Foundation grant 346300 and the Polish Government MNiSW 2015-2019 matching fund.


\end{document}